\documentclass[reqno]{amsart}

\usepackage{amsmath,amsthm, amscd, amssymb, amsfonts, mathrsfs}
\usepackage[all,cmtip]{xy}
\usepackage{color}

\usepackage[T1]{fontenc}
\usepackage{array,multirow}

\usepackage{tikz}
\usetikzlibrary{patterns}
\usetikzlibrary{arrows}
\usetikzlibrary{shapes.misc, positioning}
\usetikzlibrary{backgrounds}

\pgfdeclarepatternformonly{my crosshatch dots}{\pgfqpoint{-1pt}{-1pt}}{\pgfqpoint{5pt}{5pt}}{\pgfqpoint{6pt}{6pt}}%
{
    \pgfpathcircle{\pgfqpoint{0pt}{0pt}}{.5pt}
    \pgfpathcircle{\pgfqpoint{3pt}{3pt}}{.5pt}
    \pgfusepath{fill}
}

\usepackage{multicol}
\usepackage{enumitem}

\usepackage[plainpages=false,pdfpagelabels,
hypertexnames=false]{hyperref}

\allowdisplaybreaks

\newcommand{\cI}{\mathcal{I}}
\newcommand{\cJ}{\mathcal{J}}
\newcommand{\cO}{\mathcal{O}}
\newcommand{\cS}{\mathcal{S}}

\newcommand{\cM}{\mathcal{M}}
\newcommand{\cC}{\mathcal{C}}

\newcommand{\cG}{\mathcal{G}}

\newcommand{\Lbr}{\Lambda^{\mathbf r}}
\newcommand{\Lbp}{\Lambda^{\mathbf p}}
\newcommand{\Lbo}{\Lambda^{\mathbf o}}

\newcommand{\vt}{v_{top}}
\newcommand{\at}{\alpha_{top}}

\newcommand\dm{\mathbb D_m}
\newcommand{\yddm}{{}^{\dm}_{\dm}\mathcal{YD}}

\newcommand{\oV}{\overline{V}}
\newcommand{\oW}{\overline{W}}
\newcommand{\oZ}{\overline{Z}}
\newcommand{\oU}{\overline{U}}

\renewcommand{\_}[1]{_{\left( #1 \right)}}

\newcommand{\com}{{\Delta}}
\newcommand{\ot}{{\otimes}}
\newcommand{\ku}{\Bbbk}

\newcommand\fL{\mathsf{L}}
\newcommand\fM{\mathsf{M}}
\newcommand\fN{\mathsf{N}}

\newcommand\fR{\mathsf{R}}
\newcommand\fS{\mathsf{S}}

\newcommand\fm{\mathsf{m}}
\newcommand\fn{\mathsf{n}}

\newcommand{\Z}{{\mathbb Z}}
\newcommand{\N}{{\mathbb N}}

\newcommand{\C}{{\mathbb C}}
\newcommand{\DD}{{\mathbb D}}
\newcommand{\D}{\mathcal{D}}
\newcommand{\II}{\mathcal{I}}

\newcommand{\BV}{{\mathfrak B}}
\newcommand{\lieg}{{\mathfrak g}}
\newcommand{\lieb}{{\mathfrak b}}
\newcommand{\lien}{{\mathfrak n}}

\newcommand{\ydh}{{}^H_H\mathcal{YD}}

\newcommand{\Irr}{\operatorname{Irr}}

\newcommand{\Inf}{\operatorname{Inf}}

\newcommand{\Ind}{\operatorname{Ind}}
\newcommand{\Res}{\operatorname{Res}}

\newcommand\sgn{\operatorname{sgn}}

\newcommand\rk{\operatorname{rk}}
\newcommand\ad{\operatorname{ad}}
\newcommand\op{\operatorname{op}}

\newcommand\id{\operatorname{id}}

\newcommand\soc{\operatorname{soc}}
\newcommand\head{\operatorname{top}}

\newcommand\mm[1]{\boldsymbol{|#1\rangle}}

\newcommand\e{\varepsilon}

\theoremstyle{plain}

\newtheorem{lema}{Lemma}[section]
\newtheorem{theorem}[lema]{Theorem}
\newtheorem{cor}[lema]{Corollary}

\newtheorem{prop}[lema]{Proposition}
\newtheorem{claim}[lema]{Claim}

\newtheorem{question-app}{Question}

\theoremstyle{definition}
\newtheorem{definition}[lema]{Definition}
\newtheorem{exa}[lema]{Example}

\theoremstyle{remark}
\newtheorem{obs}[lema]{Remark}

\newtheorem{notation}[lema]{Notation}

\begin{document}

\title[]{Simple modules of small quantum groups\\
at dihedral groups}

\author[G. A. Garc\'ia, C. Vay ]{Gast\'on Andr\'es Garc\'ia and Cristian Vay}

\address{\newline\noindent Departamento de Matem\'atica, Facultad de Ciencias Exactas,
Universidad Nacional de La Plata. CMaLP -- CONICET, (1900) La Plata, Argentina.}
\email{ggarcia@mate.unlp.edu.ar}

\address{\newline\noindent Facultad de Matem\'atica, Astronom\'\i a, F\'\i sica y Computaci\'on,
Universidad Nacional de C\'ordoba. CIEM -- CONI\-CET. Medina Allende s/n, 
Ciudad Universitaria 5000 C\'ordoba, Argentina}
\email{cristian.vay@unc.edu.ar}

\thanks{\noindent 2010 \emph{Mathematics Subject Classification.} 16T20, 17B37, 22E47, 17B10.}

\thanks{The work of C. V. was partially supported by CONICET, Secyt (UNC) and Foncyt PICT 2016-3927. The work of
G. A. G. was partially supported by CONICET, Secyt (UNLP) and Foncyt PICT 2016-0147.}

\begin{abstract}
Based on previous 
results on the classification of finite-dimensional Nichols algebras over dihedral groups and the characterization of simple modules of Drinfeld doubles, 
we compute the irreducible characters of the Drinfeld doubles of bosonizations of finite-dimensional Nichols algebras over 
the dihedral groups $\DD_{4t}$ with $t\geq 3$. To this end, we develop new techniques that can be applied to Nichols algebras 
over any Hopf algebra. Namely, we explain how to construct recursively irreducible representations when the Nichols algebra is generated by a 
decomposable module, and show that the highest-weight of minimum degree in a Verma module determines its socle. 
We also prove that tensoring a simple module by a rigid simple module gives a semisimple module.
\end{abstract}

\maketitle



\section{Introduction}

This paper is devoted to study the representations
of certain families of Hopf algebras $\D(V,\dm)$, which are given by
Drinfeld doubles of bosonizations of finite-dimensional Nichols algebras $\BV(V)$ over 
dihedral groups $\dm$ of order $2m$ with $m=4t\geq 12$. The Hopf algebras $\D(V,\dm)$ might be considered as analogs of small quantum groups but with 
non-abelian torus. This election is based on the classification result in \cite{FG} of 
all finite-dimensional Nichols algebras over $\dm$. In particular, $V$ belongs to an infinite family of reducible 
Yetter-Drinfeld modules over $\dm$ and $\BV(V)\simeq\bigwedge V$.

The small quantum groups or Frobenius-Lusztig kernels $u_{q}(\lieg)$ are 
finite dimensional quotients of quantum universal enveloping algebras $U_{q}(\lieg)$ at a root of unity $q$ for $\lieg$
a semisimple complex Lie algebra \cite{L}, with some restrictions on the order $\ell$ of $q$ depending on the type of $\lieg$.
As it is well-known, $U_{q}(\lieg)$ can be described as a quotient 
of the Drinfeld double of the quantum group 
$U_{q}(\lieb)$ associated with a standard Borel subalgebra $\lieb$ of $\lieg$.
Consequently, $u_{q}(\lieg)$ can also be described as a quotient of the Drinfeld double of the small quantum 
group $u_{q}(\lieb)$. The latter is a pointed Hopf algebra over the abelian group $\Z_{\ell}^{n}$ with $n= \rk \lieg$. 
As such, it is isomorphic to the smash product or \textit{bosonization} $u_{q}(\lieb) \simeq u_{q}(\lien)\# \C \Z_{\ell}^{n}$
and $u_{q}(\lien)$ is a Nichols algebra of diagonal type \cite{Ro}, \cite{AS}, \cite{A}. In fact, 
as a consequence of the classification theorem due to Andruskiewitsch and Schneider \cite{AS2}, 
under mild conditions on the order of the groups, 
all complex finite-dimensional pointed Hopf algebras over abelian groups are variations of $u_{q}(\lieb)$. 
In these notes, we consider objects analogous to the small quantum groups $u_{q}(\lieb)$ but containing 
a non-abelian torus. These pointed Hopf algebras are given by
bosonizations of finite-dimensional Nichols algebras over the infinite family of dihedral 
groups $\dm$ classified in \cite{FG}. These Nichols algebras actually turn out to be exterior algebras over 
semisimple objects in the braided category $\yddm$. See Section \ref{sec:dihedral groups} for more details.

More generally, one may consider the Drinfeld double $\D(V,H):=\D(\BV(V)\#H)$ of the 
bosoni\-zation of a finite-dimensional Nichols algebra $\BV(V)$ over a finite-dimensional Hopf algebra $H$. 
These kind of generalized small quantum groups admits a triangular decomposition $\BV(V)\ot\D(H)\ot\BV(\oV)$ 
in the sense of Holmes and Nakano \cite{HN}, \cite{BT}. Thus, as in the classical context, the simple modules can be obtained as quotients 
of generalizations of Verma modules and consequently are classified by their \textit{highest-weights}. Here, the {\it weights} are the simple representations of 
the Drinfeld double $\D(H)$, which plays the role of the Cartan subalgebra. 
In case $H=\Bbbk \Gamma$ is a group algebra of an abelian group, the weights are one-dimensional and there are several results known, 
see for example \cite{AAMR}, \cite{Ch}, \cite{KR1, KR2}, \cite{HY}.
However, in case $\Gamma$ is not an abelian group, the weights are not necessarily one-dimensional and the computations turn out to be more involved. 
In this case, the description of the simple modules is only known for $\BV(V)=\mathcal{E}_{3}$ being the Fomin-Kirillov algebra over the symmetric group 
$\mathbb{S}_3$ \cite{PV}. It is worth pointing out that the category of graded modules can be endowed with an structure of 
\textit{highest-weight category} when $H$ is semisimple. 

In conclusion, with the classification of the irreducible representations at hand, 
the central problem to address is to compute their characters, {\it i.e.} their weight decomposition. 
Our main contribution is the solution to this problem for small quantum groups at the dihedral groups $\dm$, see Theorem \ref{thm:simple-D(M_I)-modules}. 
On the way, we develop new techniques and provide results that can be applied to Nichols algebras over any Hopf algebra.

Specifically, we develop a recursive process based on different triangular decompositions and
bosonizations when $V$ is reducible: If $V=U\oplus W$, then the Nichols algebra $\BV(V)$ can be described as a 
\textit{braided} bosonization $\BV(Z)\#\BV(U)$ for certain module $Z$ associated with the adjoint action 
of $U$ on $W$, and $\BV(V)\#H \simeq \BV(Z)\#\big(\BV(U)\#H\big)$ as Hopf algebras, see \cite{AA}, \cite{AHS}, \cite{HS}.
In this situation, we construct first the simple modules of $\D(U,H)= \D(\BV(U)\#H\big)$ from the simple $\D(H)$-modules -- the weights-- and then
use the same proceeding to construct those of $\D(V,H)$ from the former.
To use this tool, we need first to prove some technical results on composition of certain functors that assure that our recursive process gives the desired answer.

We would like to emphasize other two new results which might help to describe the simple modules of generalized small quantum groups. 
First, it is known that the simple modules can also be obtained as the socle of the Verma modules \cite[Theorem 2]{PV};
we show in Corollary \ref{cor:highest weight of the socle} that the socle of a Verma module is generated by its highest-weight of minimum degree. 
Second, a simple $\D(V,H)$-module is said to be {\it rigid} if it is also simple as $\D(H)$-module;
we prove in Theorem \ref{thm:rigid tensor} that the tensor product between a simple module and a rigid simple module is semisimple.

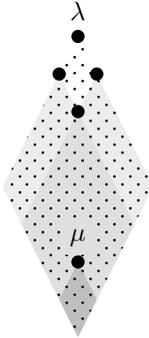
\begin{figure}[!h]
\begin{tikzpicture}

\fill[fill=gray!15] (.25,1.5) -- (1,0) -- (0,-2) -- (-.75,-.5) -- (.25,1.5);
;

\fill[fill=gray!15] (-.25,1.5) -- (-1,0) -- (0,-2) -- (.75,-.5) -- (-.25,1.5);
;

\node[circle,fill=black,inner sep=0pt, minimum width=5pt,label=above:{$\lambda$}] at (0,2) {};

\node[circle,fill=black,inner sep=0pt, minimum width=5pt] at (-.25,1.5) {};

\node[circle,fill=black,inner sep=0pt, minimum width=5pt] at (.25,1.5) {};

\fill[fill=gray!25] (0,1) -- (-.75,-.5) -- (0,-2) -- (.75,-.5) -- (0,1);

\node[circle,fill=black,inner sep=0pt, minimum width=5pt] at (0,1) {};

%


\fill[fill=gray!50] (0,-1) -- (-.25,-1.5) -- (0,-2) -- (.25,-1.5) -- (0,-1);
;

%


\fill[pattern=my crosshatch dots] (0,2) -- (1,0) -- (0,-2) -- (-1,0) -- (0,2);

\node[circle,fill=black,inner sep=0pt, minimum width=5pt,label={[fill=gray!25]above:$\mu$}] at (0,-1) {};

\end{tikzpicture}
\caption{The Verma module $\fM_H^V(\lambda)$. The (big) dots represent their (highest-)weights. 
The shadow regions indicate submodules generated by highest-weights. 
In particular, the region on the bottom is the socle $\fS^{V}_H(\lambda)$ which is generated by the highest-weight of minimum degree. 
The white region on the top depicts its unique simple quotient $\fL^{V}_H(\lambda)$.}
\label{figure:intro}
\end{figure}

As the category ${}_{\D(V,\dm)}\cM$ is non-semisimple, this is a first step towards a complete description of it. Nevertheless, 
taking into account that the main results needed lots of computations, we prefer to present the result on simple modules
first and leave the study of the indecomposable modules, tensor products and extensions for future work. 
Finally, we point out that the description of the simple modules over $\D(V,H)$ can also be seen as a first step for finding
new finite-dimensional Hopf algebras by using the generalized lifting method, see for instance \cite{AA}.

The paper is mostly self-contained and includes figures to lighten the reading. It is organized as follows. 
In the preliminaries we collect definitions, notation and basic facts that are used along the paper. 
In Section \ref{sec:general framework} we recall the general framework of Hopf algebras with 
triangular decomposition and present the new results mentioned above, cf. Corollary \ref{cor:highest weight of the socle} and
Theorems \ref{thm:rigid tensor} and \ref{thm:simples W op V}. In Section \ref{sec:dihedral groups} 
we summarize some facts about the category of $\D(\dm)$-modules. We list the simple $\D(\dm)$-modules 
and compute some tensor products between them. Dealing with these tensor products is one of the main 
issues when the corresponding weights are not one-dimensional. We also recall the classification due to \cite{FG} 
of the finite-dimensional Nichols algebras $\BV(V)$ over $\dm$, cf. Theorem \ref{thm:all-nichols-dm}. 
We fix a lighter notation to work with their Drinfeld doubles $\D(V,\dm)$ in the last section and characterize those which are spherical, 
see Theorem \ref{thm:DrinfeldDoubleSpherical}. Finally, we compute in Section \ref{sec:simple-MI-modules} 
the weight decomposition of the simple $\D(V,\dm)$-modules for all $V$ in the classification of \cite{FG}, 
see Theorem \ref{thm:simple-D(M_I)-modules}; we first consider the case when $V$ is simple in \S \ref{subsec:MIsinglecase},
and then prove the recursive step in \S\ref{subsec:MIsemisimple}.

\section*{Acknowledgements}
This work was partially done during a short stay of C.V. at UNLP as visiting professor. 
The authors would like to thank the referee for the careful reading and all the suggestions 
that helped to improve the exposition of the paper.

\section{Preliminaries}
In this section we introduce notation and recall some basic results that are needed along the paper.

\subsection{Conventions}\

We work over an algebraically closed field $\Bbbk$ of
characteristic zero. All vector spaces are considered over $\Bbbk$ and
$\ot = \ot_{\Bbbk}$.
For $n\in\N$, we denote by $\Z_n$ the ring of integers module $n$. We use the same letter to indicate an integer and its class in $\Z_n$. 
Through the work graded means $\Z$-graded. 
Let $\fN=\oplus_{n\in\Z}\fN_n$ be a graded vector space. For $i\in \Z$, 
the shift $\fN[i]$ of $\fN$ is the same vector space $\fN$ but with shifted grading, where 
$\fN[i]_n=\fN_{n-i}$ as homogeneous component of degree $n$.

We work with Hopf algebras $H$ over $\Bbbk$. As usual, we denote the comultiplication by $\com$, the 
antipode by $\cS$ and the counit by $\e$. The comultiplication 
is written using Sweedler's sigma notation without the summation symbol, \textit{i.e.}
$\Delta(h)=h\_{1}\ot h\_{2}$ for all $h\in H$. Analogously, for a left $H$-comodule $(V,\lambda)$ 
we write $\lambda(v)= v_{(-1)}\ot v_{(0)} \in H\ot V$ for all $v\in V$ to denote its coaction.
We refer to \cite{Rad} 
for basic and well-known results in the theory.

Throughout these notes, we make use of the triangular decomposition associated with 
finite-dimensional graded algebras as introduced by Holmes and Nakano \cite{HN}. Here we apply it to  
Drinfeld doubles of finite-dimensional  Hopf algebras.   
A graded algebra 
$A= \bigoplus_{n\in \Z} A_{n}$ admits a \textit{triangular decomposition} 
if there exist graded subalgebras $A^{-}$, $T$ and $A^{+}$
such that the multiplication $m: A^{-}\ot T \ot A^{+} \to A$ gives a linear isomorphism and 
\begin{enumerate}
 \item[(td1)] $A^{\pm} \subseteq \bigoplus_{n\in \Z_{\pm}} A_{n}$ and $T\subseteq A_{0}$;
 \item[(td2)] $(A^{\pm})_{0} = \Bbbk$;
 \item[(td3)] $B^{\pm}:=A^{\pm} T = T A^{\pm}$.
\end{enumerate}
In our situation, this coincides with \cite[Definition 3.1]{BT} as 
$T$ is a split $\Bbbk$-algebra because of our assumptions on $\Bbbk$. 
We denote by ${}_A\cM$ the category of finite-dimensional left $A$-modules and by 
${}_A\cG$ the category of the graded ones with morphisms preserving the grading. 
We write ${}_A\cC$ to refer to either one of these categories. 
We denote by $\Irr{}_{A}\cC$ a complete set of non-isomorphic simple objects in ${}_{A}\cC$.

\section{The general framework}\label{sec:general framework}

We outline in this section the general framework of our study. 
Eventually, we will be more explicit in the successive subsections according to our convenience. 
For more details we refer the reader to \cite{AS}, \cite{BT} and \cite{V}.

\subsection{Drinfeld doubles of bosonization of Nichols algebras}\label{subsec:Boson-Nichols-Drinfeld}\
 Fix $H$ a finite-dimensional Hopf algebra. 

\subsubsection{Nichols algebras} \label{subsubsec:Def-Nichols} 

A left Yetter-Drinfeld module over $H$ is a left 
$H$-module $(V,\cdot)$ and a left $H$-comodule $(V,\lambda)$ that satisfies the compatibility condition
$$ 
\lambda(h\cdot v) = h_{(1)}v_{(-1)}\cS(h_{(3)})\ot h\_{2}\cdot v_{(0)}\qquad \text{ for all }h\in H, v\in V.
$$
The finite-dimensional left Yetter-Drinfeld modules over $H$ together with morphisms of left $H$-modules and left $H$-comodules 
form a braided rigid tensor category denoted by $\ydh$.
The braiding is given by $c_{V,W}(v\ot w) = v_{(-1)}\cdot w \ot v_{(0)}$ for all $v\in V$, $w\in W$ with $V$, $W$ objects in $\ydh$.

Let $V \in \ydh$. Then, the tensor algebra $T(V)$ is a graded braided
Hopf algebra in $\ydh$. Roughly speaking, it satisfies the axioms of a Hopf algebra but the structural maps are morphism in the category. 
For instance, its comultiplication is determined by setting 
$\com(v) = v\ot 1 + 1 \ot v$ for all $v\in V$. 

The \textit{Nichols algebra} $\BV(V) = \bigoplus_{n\geq 0} \BV^{n}(V)$ of $V$ is the graded braided Hopf algebra in $\ydh$ defined by the quotient
$\BV(V) = T(V) /\cJ(V)$,
where $\cJ(V)$ is the largest Hopf ideal of $T(V)$ generated as an ideal by homogeneous elements of degree greater
than or equal to $2$. By definition, we have that 
$\BV^{0}(V) = \Bbbk$ and $\BV^{1}(V) = V$.
In case $\BV(V)$ is finite-dimensional, we denote by $n_{top}$ its maximum degree. It is well-known that 
$\lambda_V:=\BV^{n_{top}}(V)$ is one-dimensional (in fact, $\BV(V)$ satisfies the Poincar\'e duality);
in particular, it is a simple $H$-module and $H$-comodule.
A fixed linear generator $v_{top} \in \BV^{n_{top}}(V)$ is usually called a \textit{volume element}.
We refer to \cite{A} for more details on Nichols algebras.

\subsubsection{}\label{subsec:Bosonization}

The {\it bosonization} of $\BV(V)$ over $H$ is
a usual Hopf algebra whose underlying 
vector space is 
\begin{align*}
\BV(V)\#H:=\BV(V)\ot H 
\end{align*}
and it is endowed with a Hopf algebra structure which is a sort of semidirect product. 
It is generated by $V$ and $H$ as an algebra, whereas its multiplication and 
comultiplication are completely determined by
\begin{align*}
hv& =   (h_{(1)}\cdot v )\# h_{(2)}\quad\mbox{and}\quad
\Delta(v)= v\ot1+v\_{-1}\ot v\_{0}
\end{align*}
for all $v\in V$ and $h\in H$. 
In particular, $H=1\#H$ is a Hopf subalgebra and 
$\BV(V)=\BV(V)\#1$ is a subalgebra which coincides with the subalgebra of left coinvariants
associated with the projection $\BV(V)\#H \to H$ given by $\pi(b\#h) = \varepsilon(b)h$ for all 
$b\in \BV(V)$ and $h\in H$.
Note that the adjoint action of $H$ on $V$ coincides with the action as Yetter-Drinfeld module. 
That is, for all $h\in H$ and $v\in V$ we have that 
$\ad(h)(v)=h\_{1}v\cS(h\_{2})=(h\_{1}\cdot v)\#h\_{2}\cS(h\_{3})=h\cdot v$.

\subsubsection{}\label{subsec:Drinfeld double} 
The {\it Drinfeld double of $H$} is the Hopf algebra defined on the vector space 
\begin{align*}
\D(H):=H\ot H^*
\end{align*}
in such a way that $H=H\ot1$ and $H^{*\op}=1\ot H^{*\op}$ are Hopf subalgebras,
and the elements $h\in H$ and $f\in H^{*}$ obey the multiplication rule
\begin{align*}
fh=\langle f\_{1},h\_{1}\rangle\langle f\_{3},\cS(h\_{3})\rangle h\_{2}f\_{2}.
\end{align*}
By convention we write $hf=h\ot f$ for all $h\in H$, $f\in H^{*}$, cf. \cite[Theorem 7.1.1]{M}. 

The Drinfeld double $\D(H)$ is a quasitriangular Hopf algebra with $R$-matrix given by
$R=\sum_i f_i\ot h_i \in \D(H)\ot \D(H)$, where $\{h_i\}\subset H$ and $\{f_i\}\subset H^*$ are 
dual bases of $H$ and $H^{*}$, respectively. Also, it holds that $R^{-1}=\sum_i \cS(f_i)\ot h_i$. 
Thus, the category $_{\D(H)}\cM$ of $\D(H)$-modules is braided with braiding 
$c_{M,N}(m\ot n)=\sum_i (h_i\cdot n)\ot (f_i\cdot m)$ for all $m\in M$, $n\in N$ with $M,N\in{}_{\D(H)}\cM$. 
As the braiding is invertible, one may consider also ${}_{\D(H)}\cM$  as braided category with braiding $c^{-1}$.

It is possible to relate the categories $\ydh$, $_{\D(H)}\cM$ and $_{\D(H)}^{\D(H)}\mathcal{YD}$ by means of the following functors:
\begin{align*}
\xymatrix{
\ydh
\ar@{->}^{F}[rr]
&&
\left(_{\D(H)}\cM,c\right)
\ar@{->}^{F_R}[rr]
&&
_{\D(H)}^{\D(H)}\mathcal{YD}\\
_{H^{*\op}}^{H^{*\op}}\mathcal{YD}
\ar@{->}^{\bar{F}}[rr]
&&
\left(_{\D(H)}\cM,c^{-1}\right)
\ar@{->}_{F_{R^{-1}}}[rru]
&&
}
\end{align*}
Explicitly, $F$ is the braided equivalence which transforms $M\in \ydh$ 
into a $\D(H)$-module with  the action
\begin{align*}
(hf)\cdot m=\langle f,m\_{-1}\rangle h\cdot m\_{0},
\end{align*}
for all $h\in H$, $f\in H^*$ and $m\in M$. Analogously, $\bar{F}$  
transforms $M\in{}_{H^{*\op}}^{H^{*\op}}\mathcal{YD}$ into a $\D(H)$-module with the action
\begin{align*}
(hf)\cdot m=\langle (f\cdot m)\_{-1},\cS(h) \rangle (f\cdot m)\_{0}.
\end{align*}

The functors $F_R$ and $F_{R^{-1}}$ 
are both fully faithful and are defined as follows: 
For $M \in\, _{\D(H)}\cM $, the object $F_R(M)$ (resp. $F_{R^{-1}}(M)$) coincides with $M$ as $\D(H)$-module, 
whereas the $\D(H)$-coaction
is provided by the action of the $R$-matrix $R$ (resp. $R^{-1}$) as follows
\begin{equation}\label{eq:coactionsD(H)}
m\_{-1}\ot m\_{0}=\sum_ih_i\ot(f_i\cdot m)\quad\left(\mbox{resp.}\quad \sum_{i}\cS(f_i)\ot(h_i\cdot m)\right) 
\end{equation}
for all $m\in M$. 

Note that the braidings of $V$ and $F_R\circ F(V)=V$ coincide as linear maps. 
This implies that the Nichols algebras $\BV(V)$ and $\BV(F_R\circ F(V))$ are isomorphic as braided Hopf algebras \cite{T}. 
In particular, they are isomorphic as algebras and coalgebras. In the sequel, we identify both Nichols algebras
and consider $\BV(V) \in{}_{\D(H)}^{\D(H)}\mathcal{YD}$.

Let $V^*$ be the $\D(H)$-module dual to $V$. We define
\begin{align*}
\oV= F_{R^{-1}}(V^*)\in{}_{\D(H)}^{\D(H)}\mathcal{YD}. 
\end{align*}

As above, the Nichols algebras 
$\BV(\oV)$ in ${}_{\D(H)}^{\D(H)}\mathcal{YD}$ and 
$\BV(V^{*})$ in $({}_{\D(H)}\cM,c^{-1})$ are isomorphic as braided Hopf algebras. 
Besides, there is an isomorphism of $\D(H)$-module algebras
\begin{align}\label{eq:nichols de oV}
\BV(\oV)\simeq\BV(F_R(V^*))
\end{align}
where the algebra on the right-hand side is the Nichols algebra of $F_R(V^*)$ in 
${}_{\D(H)}^{\D(H)}\mathcal{YD}$, or equivalently the Nichols algebra of $V^*$ in $({}_{\D(H)}\cM,c)$. 
This follows from \cite[Lemma 1.11]{AHS}, see also \cite[(4.9)]{V}.

Actually, one may consider $V^{*}\in\, _{H^{*\op}}^{H^{*\op}}\mathcal{YD}$  with action and coaction defined by
\begin{align*}
\langle f\cdot \alpha,x\rangle=\langle f,\cS^{-1}(x\_{-1})\rangle\langle \alpha,x\_{0}\rangle\quad\mbox{and}\quad
\langle \alpha,h\cdot x\rangle=\langle \alpha\_{-1},h\rangle\langle \alpha\_{0},x\rangle
\end{align*}
for all $f\in H^{*op}$, $h\in H$, $x\in V$ and $\alpha\in V^*$. Then 
$\bar{F}(V^*)$ is the dual object of $V$ as $\D(H)$-module and hence $\oV=F_{R^{-1}}\circ\bar{F}(V^*)$.

The Nichols algebras $\BV(\oV)$ and $\BV(V)$ in ${}_{\D(H)}^{\D(H)}\mathcal{YD}$ play a central role in the rest of the paper. 
In case they are finite-dimensional, they 
are related by the fact that there is an isomorphism of Hopf algebras $(\BV(V)\# H)^{*\op} \simeq \BV(\oV)\#H^{*\op}$ 
(adapt the proof of \cite[Lemma 5]{PV}).

\subsubsection{A generalized quantum group} \label{subsub:generalized qg}

From now on, we denote by $\D(V, H)$ the Drinfeld double of $\BV(V)\#H$ with $\BV(V)$ a finite-dimensional 
Nichols algebra in $\ydh$. We 
describe its Hopf algebra structure following \cite[Lemma 4.3]{V}. 

From the very definition of the Drinfeld double, it is possible to show that
$$
\D(V, H)\simeq\BV(V)\ot H\ot(\BV(V)\ot H)^*\simeq\BV(V)\ot H\ot\BV(\oV)\ot H^*
$$
as vector spaces. 
Via this isomorphism, we may assume that $\D(V ,H)$ is generated as an 
algebra by the elements of $V$, $\oV$, $H$ and $H^*$. Henceforth, 
the Hopf algebra structure of $\D(V ,H)$ is completely determined by the following features:
\begin{enumerate}[label=\rm{(\alph*)}]
 \item The subalgebra generated by $H$ and $H^*$ is a Hopf subalgebra isomorphic to $\D(H)$.
 \item  The subalgebra generated by $V$ (resp. $\oV$) and $\D(H)$ is isomorphic to the bosonization 
 \begin{align*}
 \D^{\leq0}(V, H):=\BV(V)\#\D(H)\quad\left(\mbox{resp.}\quad\D^{\geq0}(V, H):=\BV(\oV)\#\D(H)\right).
 \end{align*}
In particular, both $\D^{\leq0}(V, H)$ and $\D^{\geq0}(V, H)$ are Hopf subalgebras.
 \item  If $v\in V$ and $\alpha\in\oV$, then
 \begin{align}\label{eq:commutation rules gral}
 \alpha v= (\alpha\_{-1}\cdot v)\alpha\_{0}+\Phi_{\alpha,v}
 \end{align}
where $\Phi_{\alpha,v}=\langle \alpha,v\rangle-\alpha\_{-1}v\_{-1}\langle \alpha\_{0},v\_{0}\rangle  \in\D(H)$
\footnote{This formula is a revised version of \cite[$(4.5)$]{V}, see Equation (4.5) of arXiv:1808.03799v2.}.
 %
\end{enumerate}

It turns out that $\D(V ,H)$ is a graded Hopf algebra with grading determined by
$\deg V=-1$, $\deg\oV = 1$ and 
$\deg\D(H)=0$. Moreover, there is an isomorphism of graded vector spaces induced by the multiplication: 
\begin{align}\label{eq:triangular decomsposition of D}
\D(V ,H)\simeq\BV(V)\ot\D(H)\ot\BV(\oV). 
\end{align}
In conclusion, $\D(V ,H)$ admits a triangular decomposition with $T = \D(H)$, $A^{-} = \BV(V)$ and $A^{+} = \BV(\oV)$.

\subsection{Simple \texorpdfstring{$\D(V,H)$}{D(V,H)}-modules}\label{subsec:simples}\

In this subsection we explain how to compute the simple modules of $\D(V ,H)$ 
by exploiting its triangular decomposition. 
For more details we refer to \cite{BT,PV,V}.

We begin by constructing the proper standard modules which are the images of the composition of the functors
$$
\xymatrix{\fM_{H}^{V}: {}_{\D(H)}\cC \ar[rr]^{\Inf_{\D(H)}^{\D^{\geq0}(V, H)}} && {}_{\D^{\geq0}(V, H)}\cC\ar[rr]^{\Ind_{\D^{\geq0}(V, H)}^{\D(V,H)}} 
& &
{}_{\D(V, H)} \cC}
$$
where $\Inf_{\D(H)}^{\D^{\geq0}(V, H)}$ is given by the canonical (graded)  
Hopf algebra epimorphism $\D^{\geq0}(V, H)\twoheadrightarrow\D(H)$ and 
$\Ind^{\D(V,H)}_{\D^{\geq0}(V, H)}$ by the inclusion $\D^{\geq0}(V, H)\hookrightarrow\D(V,H)$. 
More explicitly, the \textit{proper standard module} of $N\in{}_{\D(H)}\cC$ is
$$
\fM_{H}^{V}(N) = \D(H,V) \ot_{\D^{\geq0}(V, H)} \Inf_{\D(H)}^{\D^{\geq0}(V, H)}(N) = 
\D(H,V) \ot_{\D^{\geq0}(V, H)} N.
$$
Notice that $\oV$ acts by zero on $N=1\ot N\subset\fM_{H}^{V}(N)$.

Composing $\fM^{V}_{H}$ with the endofunctor given by taking the head, one has the functor
$$
\xymatrix{
{}_{\D(H)}\cC \ar[r]^{\fL^{V}_{H}} & {}_{\D(V,H)}\cC, & \fL^{V}_{H}(N)   = \head(\fM^{V}_{H}(N)),
}
$$
whose image is the maximal semisimple quotient of $\fM^{V}_{H}(N)$. Note that both $\fM_H^V$ and $\fL_H^V$ commute with the shift-of-grading functors.

The proper standard modules of simple objects in ${}_{\D(H)}\cC$ will play a special role in the classification of those in ${}_{\D(V,H)}\cC$. 
For that reason, we introduce a particular terminology.  We call \textit{weights} the elements in $\Irr{}_{\D(H)}\cC$.
If $\lambda\in\Irr{}_{\D(H)}\cC$ is a subobject of $\fN\in{}_{\D(V,H)}\cC$ such that $\oV\cdot \lambda = 0$, 
we say that $\lambda$ is a {\it highest-weight (of $\fN$)}. 
An object in ${}_{\D(V,H)}\cC$  generated 
by a highest-weight is called a {\it highest-weight module}. We define {\it lowest-weight (modules)} analogously using $V$ 
instead of $\oV$. Given $\lambda\in\Irr{}_{\D(H)}\cC$, we call \textit{Verma module}
\footnote{We change slightly the notation of the Verma modules with respect to \cite{V} to put the emphasis on 
$V$ and $H$ as this will be useful for our recursive argument for $V$ decomposable.}
the proper standard module
\begin{align}\label{eq:verma}
\fM_{H}^{V}(\lambda)=\D(V ,H)\ot_{\D^{\geq0}(V, H)}\lambda.
\end{align}
Thus, any highest-weight module is a quotient of a Verma module.

A classification of the simple modules over algebras with triangular decomposition is well-known, see for instance \cite{HN,BT}. 
In the case of $\D(V,H)$ this is given as follows.

\begin{theorem}\label{thm:pogo-vay}
\
\begin{enumerate}[label=\rm{(\alph*)}]
 \item $\fL^{V}_{H}(\lambda)$ is a simple highest-weight module for all $\lambda\in\Irr{}_{\D(H)}\cC$.
 \item Any 
 simple object in ${}_{\D(V,H)}\cC$ is isomorphic to $\fL^{V}_{H}(\lambda)$ for a unique weight $\lambda\in\Irr{}_{\D(V,H)}\cC$. 
 \end{enumerate}
\qed
\end{theorem}


We list some general remarks about simple modules that will be useful later. 

\begin{obs}\label{obs:highest weight module}
If $\fN$ is a highest-weight module of weight $\lambda$, then $\fN$ is a quotient of $\fM^{V}_{H}(\lambda)$ 
and the head of $\fN$ is isomorphic to $\fL^{V}_{H}(\lambda)$. This is a direct consequence of the
description above.
\end{obs}

\begin{obs}\label{obs:simple simple}
Let $B$ and $\overline{B}$ be linear bases of $V$ and $\oV$, respectively. As the elements of $V$ and $\oV$ act nilpotently,
it holds that $\fL^{V}_{H}(\lambda)=\lambda$, with $V$ and $\oV$ acting trivially, if and only if $\Phi_{\alpha,v}$ acts by zero on 
$\lambda$ for all $v \in B$ and $\alpha \in \overline{B}$. These simple modules are called {\it rigid} \cite{BT}.

\end{obs}

\begin{obs}\label{obs:verma simple}
Let us fix non-zero homogeneous elements $\vt\in\BV^{n_{top}}(V)$ and $\at\in\BV^{n_{top}}(\oV)$; 
note the maximum degree of these Nichols algebras is the same. 
By the triangular decomposition of $\D(V ,H)$, there exists $\Theta\in\D(H)$ such that
\begin{align}\label{eq:discriminante}
\at\vt-\Theta\in\oplus_{n>0}\BV^n(V)\ot\D(H)\ot\BV^n(\oV).
\end{align}
It holds that
\footnote{This follows from the proof of \cite[Corollary 15]{PV}, as it holds for any Hopf algebra $H$.}
$\fL^{V}_{H}(\lambda)=\fM^{V}_{H}(\lambda)$ if and only if $\Theta\cdot(1\ot\fm)\neq0$ for some 
$\fm\in\lambda$.
In such case, these Verma modules are also projective, see 
\cite[Corollary 5.12]{V}.
\end{obs}


\subsubsection{The character of the simple modules}\

\smallbreak

For any $N\in{}_{\D(H)}\cM$,  the proper standard module $\fM_{H}^{V}(N)$ in the category ${}_{\D(V,H)}\cM$ inherits the grading 
afforded by $\D(V,H)$. Explicitly,
\begin{align*}
\fM_{H}^{V}(N)=\bigoplus_{k=0}^{-n_{top}}\big(\fM_{H}^{V}(N)\big)_k\quad\mbox{with}\quad \big(\fM_{H}^{V}(N)\big)_k=\BV^{-k}(V) \ot N.
\end{align*}
It follows from \cite[Proposition 3.5]{GG} that the head of $\fM^{V}_{H}(N)$ in ${}_{\D(V,H)}\cM$ is 
a graded quotient. Moreover, $\fL_H^V$ commutes with the grading-forgetful functors. 
In particular, for any $\lambda\in\Irr{}_{\D(H)}\cM$ we have a decomposition
\begin{align*}
\fL^{V}_{H}(\lambda)=\bigoplus_{n\leq 0} \fL^{V}_{H}(\lambda)_{n}
\end{align*}
making it an object in ${}_{\D(V,H)}\cG$.

On the other hand, since $\D(H)$ is concentrated in degree $0$, we can consider each $\lambda\in\Irr{}_{\D(H)}\cM$ inside  ${}_{\D(H)}\cG$ 
as an object concentrated in degree $0$. Thus, $\Irr{}_{\D(H)}\cM\times\Z$ is in bijection with $\Irr{}_{\D(H)}\cG$ via $(\lambda,n)\leftrightarrow\lambda[n]$. 

Assume now $H$ is semisimple; hence $\D(H)$ also is. 
Then for each $n$, $\fL^{V}_{H}(\lambda)_{n}\simeq\oplus_{\mu\in\Irr{}_{\D(H)}\cM}\,\mu^{\oplus t_{\mu,n}}$ as $\D(H)$-modules. 
We call {\it the character of $\fL_H^V(\lambda)$} the graded $\D(H)$-module
\begin{align*}
\Res\big(\fL^{V}_{H}(\lambda)\big):=\bigoplus_{\substack{ n\leq 0 \\\mu\in\Irr{}_{\D(H)}\cM}}\mu[n]^{\oplus t_{\mu,n}}.
\end{align*}
This gives us good information about the simple modules as it is a complete invariant, 
since $\Res\big(\fL^{V}_{H}(\lambda)\big)=\lambda\oplus\mbox{weights in degree\,<0}$.


\subsubsection{The action on the Verma modules}\label{subsub:verma}\

\smallbreak

By the paragraphs above, one immediately realizes that to describe a simple module $\fL^{V}_{H}(\lambda)$
one has to deal with the submodules of the Verma module $\fM^{V}_{H}(\lambda)$. 
For explicit computations, it is convenient to keep in mind  the following key facts.
\begin{enumerate}[label=\rm{(\alph*)}]
 \item\label{item:subsub:verma a} The action 
 of $\BV(V)$ and $\D(H)$  on $\fM^{V}_{H}(\lambda)$ is given by
\begin{align*}
z\cdot(v\ot\fm)=(zv)\ot\fm\quad\mbox{and}\quad h\cdot(v\ot\fm)=(h\_{1}\cdot v)\ot (h\_{2}\cdot \fm)
 \end{align*}
 for all $z,v\in\BV(V)$, $\fm\in\lambda$ and $h\in\D(H)$. In particular,
 \begin{align}\label{eq:verma iso}
 \fM^{V}_{H}(\lambda)\simeq\BV(V)\ot\lambda=\bigoplus_{n=0}^{n_{top}}\BV^n(V)\ot\lambda
 \end{align}
 is an isomorphism and a decomposition as $\D(H)$-modules, respectively.
 
 \item\label{item:subsub:verma V oV degre} Let 
 $\fM^{V}_{H}(\lambda)_{k}=\BV^{-k}(V)\ot\lambda$. Then 
 \begin{align*}
 V\fM^{V}_{H}(\lambda)_{k}=\fM^{V}_{H}(\lambda)_{k-1}\quad\mbox{and}\quad \oV
 \fM^{V}_{H}(\lambda)_{k} \subseteq \fM^{V}_{H}(\lambda)_{k+1}.\end{align*}
 
\item To compute the action of $\BV(\oV)$ we use the commutation rule \eqref{eq:commutation rules gral}.
\item
The action of $\oV$ on a $\D(V ,H)$-module $\fN$ is a morphism of $\D(H)$-modules:
\begin{align}\label{eq:action map}
h(\alpha\fm)=(h\_{1}\cdot\alpha)(h\_{2}\fm) 
\end{align}
for all $\alpha\in\oV$, $\fm\in\fN$ and $h\in\D(H)$, cf. \cite[$(31)$]{PV}.
\end{enumerate}


\subsubsection{The simple \texorpdfstring{$\D(V,H)$}{D(V,H)}-modules as socles of the Verma modules}\

\smallbreak

We introduce now the functor
$$
\xymatrix{{}_{\D(H)}\cC \ar[r]^{\fS^{V}_{H}} & {}_{\D(V,H)}\cC, & \fS^{V}_{H}(N) = \soc(\fM^{V}_{H}(N)),
}
$$
{\it i.e.} $\fS^{V}_{H}(N)$ is the maximal semisimple submodule of 
$\fM^{V}_{H}(N)$. As in the case of the head, it follows from \cite[Proposition 3.5]{GG} that the socle of a 
standard module is a graded submodule even if $N\in{}_{\D(H)}\cM$, considered as a graded module
concentrated in degree 0. Also, $\fS^{V}_{H}$ 
commutes with the grading-forgetful functors and the shift-of-grading  functors. 

The socles of the Verma modules give us another classification of the simple modules over $\D(V,H)$. 
The next result for $H$ being a group algebra is in \cite{PV}. 

\begin{theorem}\label{thm:pogo-vay soc}
\
\begin{enumerate}[label=\rm{(\alph*)}]
 \item\label{item:thm:pogo-vay soc a}
 $\fS^{V}_{H}(\lambda)$ is a simple lowest-weight module with lowest-weight is $\lambda_V\lambda$ for all $\lambda\in\Irr{}_{\D(H)}\cC$
 \item\label{item:thm:pogo-vay soc b} Any 
 simple object in ${}_{\D(V,H)}\cC$ is isomorphic to $\fS^{V}_{H}(\lambda)$ for a unique weight $\lambda\in\Irr{}_{\D(V,H)}\cC$. 
 \end{enumerate}
\end{theorem}

\begin{proof}
From \S\ref{subsub:verma} \ref{item:subsub:verma a}, we see that the socle of $M_H^V(\lambda)$ in ${}_{\D^{\leq0}(V,H)}\cC$ is $\BV^{n_{top}}(V)\ot\lambda$ and this is simple and isomorphic to $\Inf_{\D(H)}^{\D^{\leq0}(V, H)}(\lambda_V\lambda)$. This implies \ref{item:thm:pogo-vay soc a}.

For \ref{item:thm:pogo-vay soc b}, we first consider the category ${}_{\D(V,H)}\cM$ in which the number of non-isomorphic simple modules is $\#\Irr{}_{\D(H)}\cM$ by Theorem \ref{thm:pogo-vay}. In \ref{item:thm:pogo-vay soc a} we have found the same number of non-isomorphic simple modules as tensoring by $\lambda_V$ gives a bijection on $\Irr{}_{\D(H)}\cM$. Therefore any simple module in ${}_{\D(V,H)}\cM$ is isomorphic to $\fS_H^V(\lambda)$ for a unique $\lambda\in\Irr{}_{\D(H)}\cM$. 

Let now $S\in{}_{\D(V,H)}\cG$ be a simple object and $F$ the grading-forgetful functor. Then $FS\simeq\fS^V_H(\lambda)$ for a unique $\lambda\in\Irr{}_{\D(H)}\cM$ by the above paragraph and hence $S\simeq\fS^V_H(\lambda)[n]\simeq\fS^V_H(\lambda[n])$ for some $n\in\Z$; the first isomorphism is consequence of \cite[Theorem 4.1]{GG}. This proves \ref{item:thm:pogo-vay soc b} for ${}_{\D(V,H)}\cG$ and completes the proof.
\end{proof}

Naturally, the socle of a Verma module is isomorphic to a simple highest-weight module. 
We can determine its highest-weight as follows. The next result is very useful to compute the simple modules.

\begin{cor}\label{cor:highest weight of the socle}
Let $\lambda\in\Irr{}_{\D(H)}\cC$ be a weight and
\begin{align*}
n=\min\left\{k\in\Z\mid\mbox{ there is a highest-weight $\mu$ in }\big(\fM^V_H(\lambda)\big)_k\right\}.
\end{align*}
Then $\fM^V_H(\lambda)$ has a unique highest-weight $\mu$ in degree $n$ and $\fS_H^V(\lambda)\simeq\fL_H^V(\mu)[n]$.
\end{cor}

\begin{proof}
As we mentioned, the socle of $\fM_H^V(\lambda)$ is a graded submodule, then  
$\fS_H^V(\lambda)\simeq\fL_H^V(\mu)[k]$ for some homogeneous highest-weight 
$\mu\subset\big(\fM^V_H(\lambda)\big)_k$. Let $\nu\subset\big(\fM^V_H(\lambda)\big)_j$ 
be another homogeneous highest-weight. 
The submodule generated by $\nu$ contains the socle. In particular, $\mu\subset\D(V,H)\nu$. 
Also, by the triangular decomposition, $\D(V,H)\nu=\BV(V)\nu=\nu+\sum_{i>0}\BV^{i}(V)\nu$. 
Then \S\ref{subsub:verma} \ref{item:subsub:verma V oV degre} 
implies that either $j>k$ or $j=k$ and $\nu=\mu$, and the corollary follows.
\end{proof}

\subsubsection{Example}\label{subsub:application 2}\ 

\smallbreak

Exterior algebras are
examples of Nichols algebras. They arise when the braiding of $V$ is $-flip$, 
that is $c_{V,V}(v\ot w)=-w\ot v$ for all $v,w\in V$.  We explain here a general strategy 
that applies to exterior algebras of two-dimensional vector spaces. In \S\ref{subsub:application 2n}, 
we consider any even dimensional vector space, as the Nichols algebras appearing in the 
context of the dihedral groups $\dm$ are all exterior algebras 
of vector spaces of even dimension.

Fix $V\in\ydh$ a two-dimensional module with basis $\{v_+,v_-\}$ and braiding $-flip$. 
Then $\BV(V)\simeq \bigwedge V=\ku\oplus V\oplus\ku\{v_{top}\}$ with $v_{top}=v_+v_-$. Let $\lambda$ be a weight with $V\ot\lambda$ 
semisimple and let $\mu$ be the highest-weight of minimum degree in $\fM_H^V(\lambda)$, recall Corollary \ref{cor:highest weight of the socle}. 
We have the following three  possibilities:
\begin{itemize}
 \item If $\deg\mu=0$, then $\mu=\lambda$ and $\fM_H^V(\lambda)=\fL_H^V(\lambda)=\fS_H^V(\lambda)$ is simple projective. 
 This occurs when $\Theta$ acts non-trivially on $\lambda$, see Remark \ref{obs:verma simple}. 
 One can compute $\Theta$ using $v_{top}$ and $\alpha_{top}=\alpha_+\alpha_-$ where $\{\alpha_+,\alpha_-\}$ is a basis of $\oV$.

\item If $\deg\mu=-1$, then $\Res\big(\fS_H^V(\mu)\big)=\mu[-1]\oplus\lambda_V\lambda[-2]$. 
We depict this situation in Figure \ref{fig:V dim 2}. To find $\mu$, first one has to decompose 
$V\ot\lambda$ as a direct sum of weights and then determine which is annihilated by $\oV$.

\item If $\deg\mu=-2$, then $\mu=\lambda_V\lambda$ and $\fS_H^V(\lambda)=\fL_H^V(\lambda_V\lambda)[-2]$ is a rigid module. 
To find the rigid module, one can use Remark \ref{obs:simple simple}. For that, one should compute the four elements 
$\Phi_{\pm,\pm}$ associated with the elements $v_{\pm}$ and $\alpha_{\pm}$.
\end{itemize}

\begin{figure}[!h]
\begin{tikzpicture}
\begin{scope}[on background layer]
\draw[dotted] (-4.5,1) -- (5.9,1) node at (6.15,1) {$0$};

\draw[dotted] (-4.5,0) --  (6,0) node at (6,0) {$-1$} ;

\draw[dotted] (-4.5,-1) -- (6,-1) node at (6,-1) {$-2$};

\draw[->] (-3,.75) to [out=18,in=162]
 node [midway,above,sloped] {$\ot\lambda$} (2,.75);
\end{scope}

\node at (-6,0) {};

\begin{scope}[xshift=2.5cm,xscale=.6]
\fill[fill=gray!50] 
(-.15,-1.15) to [out=315,in=225] 
(.15,-1.15) to
[out=45,in=-45] 
(.15,-.85) --
(-.85,.15) to
[out=135,in=45]
(-1.15,.15) to
[out=225,in=135]
(-1.15,-.15) --
(-.15,-1.15);

\node[circle,fill=black,inner sep=0pt, minimum width=5pt,label=above:{$\lambda$}] at (0,1) {};

%

\node[circle,fill=black,inner sep=0pt, minimum width=5pt] at (1.75,0) {};

\node[circle,fill=black,inner sep=0pt, minimum width=5pt] at (1,0) {};

\node[circle,fill=black,inner sep=0pt, minimum width=5pt] at (.5,0) {};

\node[circle,fill=black,inner sep=0pt, minimum width=5pt] at (.1,0) {};

\node[circle,fill=black,inner sep=0pt, minimum width=5pt] at (-.2,0) {};

\node[circle,fill=black,inner sep=0pt, minimum width=5pt,label={[fill=white,label distance=1pt]left:$\mu$}] at (-1,0) {};

\node[circle,fill=black,inner sep=0pt, minimum width=5pt,label=below:{$\lambda_V\lambda$}] at (0,-1) {};
\end{scope}

\begin{scope}[xshift=-3.5cm]

\node[circle,fill=black,inner sep=0pt, minimum width=5pt,label=above:{$\ku$}] at (0,1) {};

\node[circle,fill=black,inner sep=0pt, minimum width=5pt,label={[fill=white]left:$V$}] at (0,0) {};

\node[circle,fill=black,inner sep=0pt, minimum width=5pt,label=below:{$\lambda_V$}] at (0,-1) {};
\end{scope}

\end{tikzpicture}
\caption{The big dots represent the weights of $\BV(V)$ and $\fM_H^V(\lambda)$. 
Their degrees are indicated on the right. Those in the shadow region form the socle $\fS^{V}_H(\lambda)$ when $\deg\mu=-1$.}
\label{fig:V dim 2}
\end{figure}
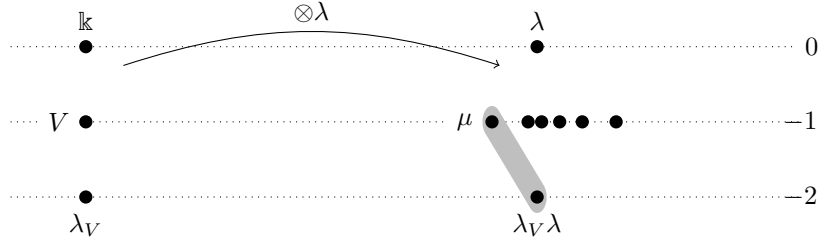

\subsubsection{Tensoring by rigids}\ 

\smallbreak

We observe that any semisimple object in ${}_{\D(V,H)}\cG$ is of the form $\fL^V_H(M):=\oplus_i\fL^V_H(\lambda_i[n_i])$ 
for some semisimple object $M=\oplus_i\lambda_i[n_i]$ in ${}_{\D(H)}\cG$.

\begin{prop}\label{prop:top-Lcero}
Let $\fN=\oplus_{i\leq t} \fN_{i}$ be an object in ${}_{\D(V,H)}\cG$ with $t \in \Z$. 
 If $\fN_{t}$ is a semisimple $\D(H)$-module and generates $\fN$
as $\D(V ,H)$-module, then 
$$
\operatorname{top}(\fN) \simeq \fL^V_H(\fN_{t}).
$$ 
\end{prop}

\begin{proof}
We have that $\oV\cdot\fN_t=0$ by the grading assumption on $\fN$. 
Then there is an epimorphism $p:\fM_H^V(\fN_t)\longrightarrow\fN$. 
Let $\fR$ be the Jacobson radical of $\fM_H^V(\fN_t)$, that is $\fM_H^V(\fN_t)/\fR\simeq\fL_H^V(\fN_t)$. 
Then $p$ induces a projection $\overline{p}:\fL_H^V(\fN_t)\longrightarrow\fN/p(\fR)$ and hence $\fN/p(\fR)$ is semisimple. 
Also, the homogeneous component of $\fR$ of degree $t$ is zero. 
Then $\overline{p}_{|\fN_t}$ is injective and therefore $\overline{p}$ so is. 
This implies $\fL_H^V(\fN_t)$ is a direct summand of $\operatorname{top}(\fN)$. 
On the other hand, $\operatorname{top}(\fN)$ is a semisimple quotient of $\fM_H^V(\fN_t)$. 
Then $\operatorname{top}(\fN)$ is a direct summand of $\fL_H^V(\fN_t)$. Thus we obtain the desired isomorphism.
\end{proof}

We now prove that tensoring a simple module by a rigid module yields a semisimple module
when $H$ is semisimple.

\begin{theorem}\label{thm:rigid tensor}
Let $\lambda,\mu\in{}_{\D(H)}\cM$ with $\fL^V_H(\mu)=\mu$ rigid. 
If $\mu\ot\lambda$ is a semisimple $\D(H)$-module,
then
\begin{align*}
\fL^V_H(\mu)\ot\fL^V_H(\lambda)\simeq\fL^V_H(\mu\ot\lambda)\simeq\fL^V_H(\lambda\ot\mu)\simeq\fL^V_H(\lambda)\ot\fL^V_H(\mu).
\end{align*}
In particular, $\fL^V_H(\mu)\ot\fL^V_H(\lambda)$ is a semisimple $\D(V ,H)$-module. Moreover, 
it is a direct sum of simple rigid modules if $\lambda$ is also rigid.
\end{theorem}

\begin{proof}
We prove only the last isomorphism, for the others follow from the fact that 
${}_{\D(H)}\cM$ and ${}_{\D(V ,H)}\cM$ are braided categories.

We start by pointing out that, as $\fL^V_H(\lambda)\ot\fL^V_H(\mu)=\fL^V_H(\lambda)\ot \mu$ is a graded $\D(V ,H)$-module, 
any simple $\D(V ,H)$-submodule $\fL^V_H(\nu)$ is a graded submodule by \cite[Proposition 3.5]{GG}.
Without loss of generality, we assume that $\lambda$ and $\mu$ are concentrated in degree $0$, since the
functors involved commute with the shift of grading. Then the homogeneous components are
$\fL^V_H(\lambda)_n\ot\mu$ for all $n\leq 0$; 
in particular, its homogeneous component of degree 0 is $\lambda\ot \mu$.

We prove first that all highest-weights are in degree 0.
Let $\fL^V_H(\nu)$ be a simple 
graded
$\D(V ,H)$-submodule of $\fL^V_H(\lambda)\ot \mu$ and 
assume $\nu\subset\fL^V_H(\lambda)_n\ot\mu$. We claim that $n=0$. Indeed, we pick 
$\sum_k u_k\ot n_k\in\nu$ with $\{n_k\}_{k\in K}$ linearly independent. For $\alpha\in\oV$, we have
\begin{align*}
0 & =\alpha\cdot\big(\sum_k u_k\ot n_k\big)=
\sum_k \big((\alpha\cdot u_k)\ot n_k+(\alpha\_{-1}\cdot u_k)\ot (\alpha\_{0}\cdot n_k)\big)\\
& =\sum_k(\alpha\cdot u_k)\ot n_k;
\end{align*}
where the first and the last equality hold because $\nu$ and $\mu$ are highest-weights, respectively.
Hence $\oV \cdot u_k=0$ for all $k$ and therefore $\oV\cdot(\D(H)\cdot u_k)=\D(H)\cdot(\oV\cdot u_k)=0$, for all $k$. 
Being $\D(H)\cdot u_k$ an graded $\D(H)$-submodule of $\fL^V_H(\lambda)$ which is annihilated by the action of $\oV$, 
it must be $\lambda$. In particular, $u_k\in\lambda$ and $\nu$ is in degree $n=0$. 

Let $\nu\subset\fL^V_H(\lambda)\ot\fL^V_H(\mu)$ be a weight in degree $0$. 
Then it is a highest-weight. Moreover, the $\D(V,H)$-submodule $\fN$ generated by $\nu$ is isomorphic to 
$\fL_H^V(\nu)$. Indeed, if $\fN=\nu\oplus_{i<0}\fN_i$ is not simple, 
there should be a highest-weight in degree $<0$ which is not possible by the paragraph above.

In conclusion, the $\D(V,H)$-submodule generated by $\lambda\ot\mu$ is semisimple and isomorphic to $\fL_H^V(\lambda\ot\mu)$. 
We prove next that $\lambda\ot\mu$ actually generates the whole module $\fL^V_H(\lambda)\ot\fL^V_H(\mu)$ and hence 
$\fL^V_H(\lambda)\ot\fL^V_H(\mu)\simeq\fL_H^V(\lambda\ot\mu)$ as we wanted. For that, we fix bases $\{m_j\}_{j\in J}$ and 
$\{n_k\}_{k\in K}$ of the weights $\lambda$ and $\mu$, respectively. 
Since $\mu$ is rigid, we have that $x\cdot n_k=0$ for any homogeneous element $x\in\BV(V)$ 
of degree $\geq1$. Also, because $\BV(V)\#H$ is a graded coalgebra, for such an element $x$ one may write its comultiplication by
$\Delta(x)=x\ot 1+\sum_t y_{t}\ot z_{t}$, where the elements $z_{t}$ 
are homogeneous of degree $\geq1$ for all $t$. 
Hence, for all $j\in J$ and $k\in K$ we have that 
\begin{align*}
x\cdot(m_j\ot n_k)=(x\cdot m_j)\ot n_k+\sum_{t}(y_{t}\cdot m_j)\ot(z_{t}\cdot n_k)=(x\cdot m_j)\ot n_k.
\end{align*}
Since $\fL^V_H(\lambda)$ is generated as a $\D(V ,H)$-module by the action of $\BV(V)$ on
$\{m_j\}_{j\in J}$, our assertion follows.

Lastly, if $\lambda$ is also rigid, then $\fL_H^V(\lambda\ot\mu)$ is concentrated in degree $0$. 
This implies that $\fL_H^V(\nu)=\nu$ for every weight $\nu$ of $\lambda\ot\mu$ and hence $\nu$ is rigid.
\end{proof}

\begin{obs}
The hypothesis of $\fL^V_H(\mu)$ being rigid is necessary. Otherwise, the tensor product might neither be generated in 
degree zero nor all its highest-weights be in degree zero. See for instance \cite[Theorem 4.1]{EGST} 
or \cite[Proposition 4.3]{PV2}. 
\end{obs}

\subsection{A recursive strategy for \texorpdfstring{$V$}{V} decomposable}\label{subsec:decomposable} \

\smallbreak

We assume here that $V=W\oplus U$ is decomposable as $\D(H)$-module with $W\neq0\neq U$. This situation arises when $H$ is
the group algebra of the dihedral group $\dm$. In particular, 
$\BV(W)$ and $\BV(U)$ are braided graded Hopf subalgebras of $\BV(V)$. Following \cite[\S2.3]{AA}, we set 
\begin{align*}
Z=\ad_c\BV(U)(W)
\subset\BV(V),
\end{align*}
where $\ad_c$ is the braided adjoint action of a Hopf algebra in $_{\D(H)}\cM$.
Notice that $W\subseteq Z$.
It holds that $Z$ is a Yetter-Drinfeld module over $\BV(U)\#H$ via the adjoint action and the coaction 
$(\pi_{\BV(U)\#H}\ot\id)\circ\Delta_{\BV(V)\#H}$, where $\pi_{\BV(U)\#H}$ is the natural projection on $\BV(U)\#H$.
Besides, it turns out that
\begin{align}\label{eq:DVH BZ DUH}
\BV(V)\#H\simeq\BV(Z)\#(\BV(U)\#H),
\end{align}
as Hopf algebras, see {\it loc.~ cit.} or \cite[\S 8]{HS} for details and references.

Naturally, we can apply the techniques described in the previous sections 
to the bosonization on the right hand side of \eqref{eq:DVH BZ DUH}, 
{\it i.e.} $\BV(U)\#H$ and $Z$ playing the role of $H$ and $V$, respectively. 
This gives us a new description of $\D(V,H)$ and its simple modules in terms of those over $\D(U,H)$, 
the Drinfeld double of $\BV(U)\#H$. In this sense, we have
$$
\D(Z,U,H):=\D\big(\BV(Z)\#(\BV(U)\#H)\big)\simeq\D(\BV(V)\#H)=\D(V,H).
$$
Namely, one may consider another $\Z$-grading on the Drinfeld double $\D(Z,U,H)$  
given by $-\deg Z=\deg\oZ =1$ and $\deg \D(U,H)=0$. 
Then
\begin{align}\label{eq:triangular decomposition of DWV}
\D(Z,U,H)\simeq\BV(Z)\ot\,\D(U,H)\,\ot\BV(\oZ)
\end{align}
yields a new triangular decomposition on $\D(V,H)$. Hence the simple $\D(Z,U,H)$-modules can be constructed 
from the simple $\D(U,H)$-modules as before. 
Of course, the latter can also be described by the same proceeding. 
Then, we have the functors 
$$
\xymatrix{{}_{\D(H)}\cM \ar@<1ex>[rr]^{\fM^{U}_{H}} \ar@<-1ex>[rr] _{\fL^{U}_{H}} && {}_{\D(U,H)}\cM 
\ar@<1ex>[rr]^{\fM^{Z}_{\BV(U)\#H}} \ar@<-1ex>[rr]_{\fL^{Z}_{\BV(U)\#H}}& 
&{}_{\D(Z,U,H)}\cM
}
$$
For instance, the Verma module associated with the simple module
$\fL^U_{H}(\lambda)$ is 
\begin{align}\label{eq:verma prima}
\fM^{Z}_{\BV(U)\#H}(\fL^U_{H}(\lambda))=\D(Z,U,H)\ot_{\D(U,H)\ot\BV(\oZ)}\fL^U_H(\lambda),
\end{align}
where we consider $\oZ$ acting by zero on $\fL^U_H(\lambda)$. 
Observe that $\fM^{Z}_{\BV(U)\#H}(\fL^U_{H}(\lambda))$ is not necessarily isomorphic to the Verma module $\fM^{V}_{H}(\lambda)$ 
defined in \eqref{eq:verma}. 
Nevertheless, we show that their heads are isomorphic. This allows us to 
construct the simple modules in a recursive way. 

\begin{lema}\label{le:M-composicion} Keeping the notation above, we have
 $\fM^{Z}_{\BV(U)\#H}\circ \fM^{U}_{H} \simeq \fM^{V}_{H}$.
\end{lema}

\begin{proof} 
Let $N$ be a finite-dimensional $\D(H)$-module. Then both  
$\fM^{Z}_{\BV(U)\#H}\big( \fM^{U}_{H}(N)\big)$ and $\fM^{V}_{H}(N)$ are generated by $N$ 
as $\D(V,H)$-modules. Besides, by definition 
$\overline{Z}$ and $\overline{U}$ act trivially on $N$ inside $\fM^{Z}_{\BV(U)\#H}\big( \fM^{U}_{H}(N)\big)$. 
Since $\overline{Z}$ contains $\overline{W}$ as a $\D(H)$-submodule, it follows that 
$\oV = \overline{W}\oplus \overline{U}$ also acts trivially on it. Thus, 
there exists a $\D(V,H)$-module epimorphism $\eta_{N}: \fM^{V}_{H}(N) \twoheadrightarrow \fM^{Z}_{\BV(U)\#H}\big( \fM^{U}_{H}(N)\big)$,
which is the identity on $N$. 
By \S\ref{subsub:verma} \ref{item:subsub:verma a} and \eqref{eq:DVH BZ DUH}, we have that
$$ 
\Res\big(\fM^{Z}_{\BV(U)\#H}\big( \fM^{U}_{H}(N)\big)\big) \simeq \BV(Z)\ot \BV(U)\ot N \simeq \BV(V)\ot N \simeq\Res\big(\fM^{V}_{H}(N)\big),
$$
that is both objects have the same dimension. This implies that  
$\eta_{N}$ is in fact an isomorphism. 
Moreover, as $\eta|_{N} = \id_{N}$, we see that $\eta_{Y}\circ \fM^{V}_{H}(f) = \fM^{Z}_{\BV(U)\#H}\big( \fM^{U}_{H}(f)\big) \circ \eta_{X}$ for any
morphism $f:X\to Y$ in ${}_{\D(H)}\cM$. Hence, $\eta$ defines a natural isomorphism between both functors.
\end{proof}

Since $\D(Z,U,H) \simeq \D(V,H)$ as Hopf algebras,
we may consider the Verma module $\fM^{Z}_{\BV(U)\#H}(\fL_H^U(\lambda))$ and its head $\fL^{Z}_{\BV(U)\#H}(\fL_H^U(\lambda))$  as graded $\D(V ,H)$-modules 
with the unique grading satisfying that $\deg\lambda=0$ thanks to \cite{GG}. 
We prove next that there is a commutative diagram 
\begin{align*}
\xymatrix{
\fM_H^V(\lambda)
\ar@{->>}[rr]\ar@{->>}[d]
&&\fL_H^V(\lambda)\\
\fM_{\BV(U)\#H}^Z(\fL_H^U(\lambda))
\ar@{->>}[rr]
&&\fL_{\BV(U)\#H}^Z(\fL_H^U(\lambda))\ar@{-}_{\simeq}[u]
}
\end{align*}
whose arrows are epimorphisms of $\D(V ,H)$-modules.

\begin{theorem}\label{thm:simples W op V}
Let $\lambda\in\Irr{}_{\D(H)}\cM$. Then $\fM^{Z}_{\BV(U)\#H}(\fL_H^U(\lambda))$ and $\fL_{\BV(U)\#H}^Z(\fL_H^U(\lambda))$ are both 
highest-weight $\D(V ,H)$-modules with
highest-weight $\lambda$. Moreover, 
\begin{align*}
\fL^{V}_{H}(\lambda)\simeq \fL_{\BV(U)\#H}^Z(\fL_H^U(\lambda)) 
\end{align*} 
and the neous components of $\fL^{V}_{H}(\lambda)$ and $\fL^U_{H}(\lambda)$ satisfy as $\D(H)$-modules
\begin{align*}
\fL^{V}_{H}(\lambda)_{n}=\bigoplus_{\substack{n=-i-k(j+1)\\ i,j,k\geq 0}}\BV^k(\ad_c\BV^j(U)(W))\fL^U_{H}(\lambda)_{-i}\ .
\end{align*}
\end{theorem}

\begin{proof}
By definition, we know that $\oW\subseteq\oZ$ and $\oU$ act by zero on 
$\lambda=1\ot\lambda$ inside $\fM^{Z}_{\BV(U)\#H}(\fL_H^U(\lambda))$. 
That is, $\lambda$ is a highest-weight. Besides, this weight generates $\fM^{Z}_{\BV(U)\#H}(\fL_H^U(\lambda))$, and hence also 
$\fL^{Z}_{\BV(U)\#H}(\fL_H^U(\lambda))$, as $\D(V ,H)$-module. 
For, 
\begin{align*}
\D(V ,H)\lambda&=\BV(Z)\D(U,H)\BV(\oZ)\lambda\\
&=\BV(Z)\D(U,H)\lambda=\BV(Z)\fL^U_{H}(\lambda) = \fM^{Z}_{\BV(U)\#H}(\fL_H^U(\lambda)). 
\end{align*}
Then, by the characterization of the highest-weight modules,
$\fM^{Z}_{\BV(U)\#H}(\fL_H^U(\lambda))$ and $\fL^{Z}_{\BV(U)\#H}(\fL_H^U(\lambda))$ are quotients of $\fM^{V}_{H}(\lambda)$. 
Since $\fL^{Z}_{\BV(U)\#H}(\fL_H^U(\lambda))$ is a simple module, we get
$\fL^{Z}_{\BV(U)\#H}(\fL_H^U(\lambda))\simeq\fL^{V}_{H}(\lambda)$ by Remark \ref{obs:highest weight module}.
Finally, by looking at the associated gradation, we deduce the second assertion.
\end{proof}

We have an analogous result for the socle. 

\begin{prop}\label{prop:soc W op V}
Let $\lambda\in\Irr{}_{\D(H)}\cM$. Then $\fS^{V}_{H}(\lambda)\simeq \fS_{\BV(U)\#H}^Z(\fS_H^U(\lambda))$ as $\D(V,H)$-modules.
\end{prop}

\begin{proof}
Let $\lambda_Z=\BV^{n_Z}(Z)$ and $\lambda_U=\BV^{n_U}(U)$ be the 
homogeneous components of maximum degree of $\BV(Z)$ and $\BV(U)$, respectively. 
These are one-dimensional and simple as modules over $\D(U,H)$ and $\D(H)$, respectively, 
recall  \S\ref{subsubsec:Def-Nichols}. In particular, $U\cdot\lambda_Z=0$. 
Moreover, $\lambda_Z\lambda_U\simeq\lambda_V$ as $\D(H)$-modules by \eqref{eq:DVH BZ DUH}. 

Using the triangular decomposition \eqref{eq:triangular decomposition of DWV}, 
Theorem \ref{thm:pogo-vay soc}  says that $\fS_{\BV(U)\#H}^Z(\fS_H^U(\lambda))$ 
is the unique simple lowest-weight $\D(Z,U,H)$-module with lowest-weight $\lambda_Z\ot\fS_H^U(\lambda)$. 
That is, $\lambda_Z\ot\fS_H^U(\lambda)$ is a simple $\D(U,H)$-submodule generating 
$\fS_{\BV(U)\#H}^Z(\fS_H^U(\lambda))$ and  $Z\cdot\big(\lambda_Z\ot\fS_H^U(\lambda)\big)=0$. 
In particular, $W\cdot\big(\lambda_Z\ot\fS_H^U(\lambda)\big)=0$.

Also by Theorem \ref{thm:pogo-vay soc}, $\fS_H^U(\lambda)$ is the unique simple lowest-weight 
$\D(U,H)-$module with lowest-weight $\lambda_U\lambda$. 
That is, $\lambda_U\lambda$ is a simple $\D(H)$-submodule generating $\fS_H^U(\lambda)$ and 
$U\cdot\lambda_U\lambda=0$. Hence $U\cdot\big(\lambda_Z\ot\lambda_U\lambda\big)=0$. 
In fact, if $u\in U$, then
$u\cdot\big(\lambda_Z\ot\lambda_U\lambda\big)=u\cdot\lambda_Z\ot\lambda_U\lambda+u\_{-1}\cdot\lambda_Z\ot u\_{0}\cdot(\lambda_U\lambda)=0$; 
for the first equality recall \eqref{eq:verma iso} and the second one follows from the first paragraph.

In conclusion, $\fS_{\BV(U)\#H}^Z(\fS_H^U(\lambda))$ is a simple $\D(V,H)$-module 
with lowest-weight $\lambda_Z\ot\lambda_U\lambda\simeq\lambda_V\lambda$. 
Again by Theorem \ref{thm:pogo-vay soc} it should be isomorphic to $\fS^{V}_{H}(\lambda)$ as desired.
\end{proof}

We stress that Theorems \ref{thm:pogo-vay} and \ref{thm:pogo-vay soc}, 
Corollary \ref{cor:highest weight of the socle}, Remarks \ref{obs:highest weight module}, \ref{obs:simple simple} and \ref{obs:verma simple} 
and \S\ref{subsub:verma} also apply to $\fL^{Z}_{\BV(U)\#H}(\fL_H^U(\lambda))$ and $\fM^{Z}_{\BV(U)\#H}(\fL_H^U(\lambda))$, since one may take 
$\BV(U)\#H$ and $Z$ to play the role of $H$ and $V$, respectively. 
We will make use of these remarks under these generalized hypotheses when we consider Nichols algebras over the dihedral groups $\dm$. 
In such a case we refer to them as \textit{the recursive version}.

\begin{obs}\label{obs:Z igual W}
The coaction $(\pi_{\BV(U)\#H}\ot\id)\circ\Delta_{\BV(V)\#H}$ and the adjoint action of $H$ on 
$W\subset Z$ coincide with its structure in $\ydh$. Moreover, if $c_{U,W}\circ c_{W,U}=\id_{W\ot U}$, 
then $Z=W$ and hence $\BV(Z)=\BV(W)$ \cite[Remark 2.5]{AA}. In this case, we have
\begin{align*}
\BV(V)\#H\simeq\BV(W)\#(\BV(U)\#H).
\end{align*}
On the other hand, if $W$ is a simple $\D(H)$-module, then $Z\simeq\fL^U_{H}(W)$ by \cite[Proposition 2.10]{AA}. 
Thus, $W$ is a rigid $\D(U,H)$-module if $c_{U,W}\circ c_{W,U}=\id_{W\ot U}$.
\end{obs}

\begin{obs}\label{obs:Z igual W -- homogeneous}
Assume that $Z=W=\fL_H^U(W)$ is simple and rigid; \textit{e.g.} when $c_{U,W}\circ c_{W,U}=\id_{W\ot U}$ by Remark \ref{obs:Z igual W}. 
The Nichols algebras appearing in the present work satisfy this property. 
By applying Theorem \ref{thm:rigid tensor}, we have that $W^{\ot k}$ is a direct sum of simples rigid modules and hence so is 
$\BV(W)=\BV(Z)$. Therefore $\fM_{\BV(U)\#H}^W(\fL_H^U(\lambda))$ and its head are semisimple as $\D(U,H)$-modules for any $\lambda\in\Irr{}_{\D(H)}\cM$. 
The homogeneous components of its head satisfy 
for all $n\leq 0$ that 
$$
\fL^{V}_{H}(\lambda)_{n}=\bigoplus_{k= 0}^{-n}\BV^k(W)\fL^U_{H}(\lambda)_{n+k}.
$$
Finally, we point out that the Hopf subalgebra generated by $W$ and $H$ is $\D(W,H)$.
\end{obs}

\subsubsection{A recursive example}\label{subsub:application 2n}\ 

\smallbreak

Keep the notation and the assumptions of Remark \ref{obs:Z igual W -- homogeneous}. 
Assume further that $W\in\ydh$ is a simple two-dimensional module with basis $\{w_+,w_-\}$ and braiding $-flip$. 
For instance, these hypotheses are satisfied by exterior algebras of vector spaces of even dimension; 
such is the case for $H=\Bbbk \dm$. 
Then we have that $\fL_H^U(W)=W$ is rigid and $\BV(W)=\bigwedge W=\ku\oplus W\oplus\ku\{w_+w_-\}$. 

Let $\lambda\in\Irr{}_{\D(H)}\cM$ be such that $W\ot\lambda$ is semisimple. 
We explain below how to describe the socle of $\fM_{\BV(U)\#H}^W(\fL_H^U(\lambda))$ 
using the recursive version of Corollary \ref{cor:highest weight of the socle}. See Figure \ref{fig:V dim even}.

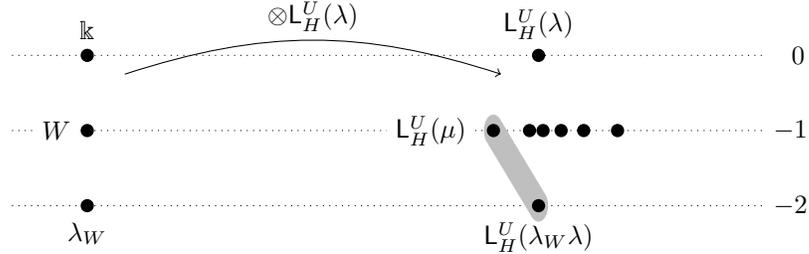
\begin{figure}[!h]
\begin{tikzpicture}    

\draw[dotted] (-4.5,1) -- (5.5,1) node at (5.95,1) {$0$};

\draw[dotted] (-4.5,0) --  (5.85,0) node[fill=white] at (5.85,0) {$-1$} ;

\draw[dotted] (-4.5,-1) -- (5.85,-1) node[fill=white] at (5.85,-1) {$-2$};

\draw[->] (-3,.75) to [out=18,in=162]
 node [midway,above,sloped] {$\ot\fL^U_H(\lambda)$} (2,.75);

\node at (-6.35,0) {};

\begin{scope}[xshift=-3.5cm]

\node[circle,fill=black,inner sep=0pt, minimum width=5pt,label=above:{$\ku$}] at (0,1) {};

\node[circle,fill=black,inner sep=0pt, minimum width=5pt,label={[fill=white]left:$W$}] at (0,0) {};

\node[circle,fill=black,inner sep=0pt, minimum width=5pt,label=below:{$\lambda_W$}] at (0,-1) {};
\end{scope}

\begin{scope}[xshift=2.5cm,xscale=.6]

\fill[fill=gray!50] 
(-.15,-1.15) to [out=315,in=225] 
(.15,-1.15) to
[out=45,in=-45] 
(.15,-.85) --
(-.85,.15) to
[out=135,in=45]
(-1.15,.15) to
[out=225,in=135]
(-1.15,-.15) --
(-.15,-1.15);

\node[circle,fill=black,inner sep=0pt, minimum width=5pt,label=above:{$\fL^U_H(\lambda)$}] at (0,1) {};

\node[circle,fill=black,inner sep=0pt, minimum width=5pt] at (1.75,0) {};

\node[circle,fill=black,inner sep=0pt, minimum width=5pt] at (1,0) {};

\node[circle,fill=black,inner sep=0pt, minimum width=5pt] at (.5,0) {};

\node[circle,fill=black,inner sep=0pt, minimum width=5pt] at (.1,0) {};

\node[circle,fill=black,inner sep=0pt, minimum width=5pt] at (-.2,0) {};

\node[circle,fill=black,inner sep=0pt, minimum width=5pt,label={[fill=white,,label distance=4pt]left:$\fL^U_H(\mu)$}] at (-1,0) {};

\node[circle,fill=black,inner sep=0pt, minimum width=5pt,label=below:{$\fL^U_H(\lambda_W\lambda)$}] at (0,-1) {};
\end{scope}

\end{tikzpicture}
\caption{The dots represent the simple $\D(U,H)$-summands of $\BV(W)$ and 
$\fM^{W}_{\BV(U)\#H}(\fL^{U}_H(\lambda))$. 
Their degrees are indicated on the right. Those in the shadow region form its socle in the case that $\deg\fL_H^U(\mu)=-1$.}
\label{fig:V dim even}
\end{figure}

First, we observe that $W\ot\fL_H^U(\lambda)\simeq\fL_H^U(W\ot\lambda)$ by Theorem \ref{thm:rigid tensor}. 
Hence, as $\D(U,H)$-module, $\fM_{\BV(U)\#H}^W(\fL_H^U(\lambda))\simeq\fL_H^U(\lambda)\oplus\fL_H^U(W\ot\lambda)\oplus\fL_H^U(\lambda_W\lambda)$ 
is semisimple. Pick the homogeneous simple $\D(U,H)$-module 
$\fL_H^U(\mu)$ of $\fM_{\BV(U)\#H}^W(\fL_H^U(\lambda))$ of minimum degree such that $\oW\cdot\fL_H^U(\mu)=0$. 
Note that it is enough to check for which homogeneous summand $\fL_H^U(\mu)$ it holds that $\oW\cdot\mu=0$, since 
the Verma module is semisimple and the action of $\oW$ is a morphism of $\D(U,H)$-modules, by the recursive version of \S\ref{subsub:verma}. 
A similar reasoning can be made using the recursive versions of Remarks \ref{obs:simple simple} and \ref{obs:verma simple}. 
That is, it is enough to check that the elements $\Phi$ act trivially on $\lambda$ (resp., $\Theta$ acts non trivially on $\lambda$) 
to conclude that they also act trivially (resp., non-trivially) on all $\fL_H^U(\lambda)$. Thus, as in \S\ref{subsub:application 2}, we have three possibilities: 
\begin{itemize}
 \item If $\deg\fL_H^U(\mu)=0$, then $\mu=\lambda$ and $\fL_H^V(\lambda)=\fM^{W}_{\BV(U)\#H}(\fL^{U}_H(\lambda))$ 
 is simple and projective as $\D(V,H)$-module.

\item If $\deg\fL_H^U(\mu)=-1$, then
the socle of $\fM_{\BV(U)\#H}^W(\fL_H^U(\lambda))$ decomposes into $\fL_H^U(\mu)[-1]\oplus\fL_H^U(\lambda_W\lambda)[-2]$ as $\D(U,H)$-module. 
To find $\mu$, one has to decompose $W\ot\lambda$ 
as a direct sum of weights and then determine the one that is annihilated by $\oW$. In this case, 
one may deduce that $W\ot\lambda=\mu\oplus\overline
{\lambda}$ for some weight $\overline
{\lambda}$ and hence $\fL^V_H(\lambda)=\fL^U_H(\lambda)\oplus\fL^U_H(\overline
{\lambda})[-1]$ as $\D(U,H)$-modules. We leave the computation for the interested reader.

\item If $\deg\fL_H^U(\mu)=-2$, then $\fL_H^U(\lambda_W\lambda)[-2]$ 
is the socle of $\fM_{\BV(U)\#H}^W(\fL_H^U(\lambda))$ and hence it is a simple $\D(V,H)$-module over which $W$ and $\oW$ act trivially.
\end{itemize}

\section{The dihedral groups framework} \label{sec:dihedral groups}\ 

From now on, we fix a natural number $m\geq12$ divisible by $4$ and an 
$m$-th primitive root of unity $\omega$. We also set $n=\frac{m}{2}$. 

The dihedral group 
of order $2m$ is presented by generators and relations by 
\begin{align*}
\DD_m=\langle x,y \mid x^2,\ y^m,\  xyxy\rangle.
\end{align*}
It has $n+3$ conjugacy classes:
$\cO_{e}=\{e\}$ with $e$ the identity, $\cO_{y^{n}}=\{y^{n}\}$, 
$\cO_{x}= \{xy^{j}:\ j \text{ even} \}$, $\cO_{xy}= \{xy^{j}:\ j \text{ odd} \}$
and $\cO_{y^{i}}=\{y^{i}, y^{-i}\}$ for $1\leq i\leq n-1$.

The algebra of functions $\ku^{\DD_m}$ is the dual Hopf algebra of $\ku\DD_m$. 
We denote by $\{\delta_t\}_{t\in\DD_m}$ the dual basis of the basis of $\ku\DD_m$
given by the group-like elements, {\it i.e.} $\delta_t(s)=\delta_{t,s}$ for all $t,s \in \dm$.
The comultiplication and the counit of these elements are 
$\Delta(\delta_t)=\sum_{s\in\dm}\delta_s\ot\delta_{s^{-1}t}$ and $\e(\delta_{t})= \delta_{t,e}$ 
for all $t\in\DD_m$, respectively.

We denote by $\D\DD_m$ the Drinfeld double of $\ku\DD_m$. Since $\ku^{\DD_m}$ is a commutative algebra, 
$\ku^{\DD_m} = \big(\ku^{\DD_m}\big)^{\op}$ and consequently $\ku^{\DD_m}$ and $\ku\DD_m$ 
are Hopf subalgebras of $\D\DD_m$. Thus, the algebra structure of $\D\DD_m$ is completely determined by
the equality
\begin{align*}
&&&&\delta_{tst^{-1}}\,t=t\,\delta_{s}&&\quad\text{ for all }  s,t\in\DD_m.
\end{align*}
In this case, the
$R$-matrix reads $R=  \sum_{t\in\DD_m}\delta_t\ot t \in \D\dm\ot \D\dm$.

\subsection{The weights of \texorpdfstring{$\D\dm$}{DDm}}\label{simple DDDm modules}\

It is well-known that the simple modules over the Drinfeld double of a group algebra are classified by 
the conjugacy classes of the group and irreducible representations of their centralizers, c.f. \cite{AG} 
and references therein. 
Namely, for $g\in\DD_m$, write $\cO_g$ for its conjugacy class and $\cC_g$ for its centralizer in $\DD_{m}$. 
Let $(U,\varrho)$ be an irreducible representation of $\cC_g$. 
The $\ku\DD_m$-module induced by $(U,\varrho)$,
\begin{align}\label{eq:Mgrho}
M(g,\varrho)=\Ind_{\cC_g}^{\DD_m}U=\ku\DD_m\ot_{\ku\cC_g}U,
\end{align}
is a $\D\DD_m$-module with the $\Bbbk^{\DD_{m}}$-action defined by  
$$f\cdot (t\,\ot_{\ku\cC_g} u)=\langle f,tgt^{-1}\rangle t\,\ot_{\ku\cC_g} u,\quad\mbox{for all $f\in\ku^{\DD_m}$, $t\in\DD_m$ and $u\in U$}.$$
Then the set $\Lambda$ consisting of  the modules 
$M(g,\varrho)$'s is a set of representative of simple $\D\dm$-modules up to isomorphism, that is
\begin{align*}
\Lambda=\Irr{}_{\D\dm}\cM.
\end{align*}
It is worth noting that a $\ku^{\DD_m}$-action on a vector space $V$ is the same as a $\DD_m$-grading.
In this sense, a left $\D\DD_{m}$-module $V$ 
(or equivalently a left Yetter-Drinfeld module over $\DD_{m}$) is a $\Bbbk\DD_{m}$-module with
a $\DD_m$-grading that is compatible with the conjugation in $\DD_m$.
In our example, we have that  $U$ is concentrated in degree $g$ and 
the $\DD_m$-degree of $t\,\ot_{\ku\cC_g} u$ in 
$M(g,\varrho)$ is $tgt^{-1}$. The action of $f\in \ku^{\DD_{m}}$ is performed via the evaluation on the degree. 
We will denote by $M[s]$ the homogeneous component of degree $s\in\dm$ of a  
$\D\DD_{m}$-module. Although this notation coincides with the shift of a grading, 
we believe that this would not confuse the reader since in the latter case $s$ is an integer and here 
is an element of $\dm$.

In the following, we recall the description of the simple $\D\dm$-modules
according to the set of conjugacy classes. We present them
by fixing a basis and by describing the action of $x$, $y$ and 
the $\DD_{m}$-grading. We use symbols like $\mm{w}$ to denote elements of a particular basis for each simple module.
For more details, see \cite{FG}.

\subsubsection{The modules $M(e,\varrho)$}\label{subsec:the mechis}
Let $e$  be the identity element in $\DD_{m}$. 
Since $\cC_{e}=\DD_{m}$, we use 
the simple representations of $\DD_m$ to describe the 
simple $\D\DD_m$-modules. 
According to the amount of conjugacy classes of $\DD_{m}$, 
these are $4$ one-dimensional, say $\chi_1$, $\chi_2$, $\chi_3$, $\chi_4$, and 
$n-1$ two-dimensional, which we denote by $\rho_{\ell}$, $1\leq \ell<n$. 
For $1\leq i\leq 4$ and $1\leq \ell \leq n-1$, the simple $\D\DD_m$-modules are:
 \smallskip
\begin{itemize}
 \item[$\triangleright$] $M(e,\chi_i)=\ku\{\mm{u_i}\}$ with
 $M(e,\chi_i)= M(e,\chi_i)[e]$  and 
 \begin{align*}
&x\cdot\mm{u_1}=\mm{u_1},& &y\cdot\mm{u_1}=\mm{u_1};&
 \\
&x\cdot\mm{u_2}=-\mm{u_2},& &y\cdot\mm{u_2}=\mm{u_2};&
 \\
&x\cdot\mm{u_3}=\mm{u_3},& &y\cdot\mm{u_3}=-\mm{u_3};&
 \\
&x\cdot\mm{u_4}=-\mm{u_4},& &y\cdot\mm{u_4}=-\mm{u_4}.& 
\end{align*}

 \smallskip
 
 \item[$\triangleright$] $M(e,\rho_\ell)=\ku\{\mm{+,\ell},\mm{-,\ell}\}$ with $M(e,\rho_\ell)= M(e,\rho_\ell)[e]$ and 
 \begin{align*}
 & x\cdot\mm{\pm,\ell}=\mm{\mp,\ell},& &y\cdot\mm{\pm,\ell}=\omega^{\pm\ell}\mm{\pm,\ell}.&
 \end{align*}
\end{itemize}
Note that  $M(e,\chi_1)$ is given by the counit of $\D\dm$. 
To shorten notation, we write $\mm{\pm}=\mm{\pm,\ell}$ when the parameter $\ell$ is clear from the context.

\subsubsection{The modules $M(y^{n},\varrho)$}\label{subsec: the mels} 
Since $m=2n$, the element $y^n$ is central in $\DD_m$. 
Therefore the simple $\D\DD_m$-modules associated with $y^n$ are given by the simple representations of $\DD_m$.
As $\DD_{m}$-modules they coincide with the ones given in \S\ref{subsec:the mechis}, but these 
are concentrated in degree $y^{n}$ instead of $e$.
Explicitly, for $1\leq i\leq 4$ and $1\leq\ell\leq n-1$ these are
 \smallskip
\begin{itemize}
 \item[$\triangleright$]  $M(y^n,\chi_i)=\ku\{\mm{u_i,n}\}$ 
 with $M(y^n,\chi_i) = M(y^n,\chi_i)[y^{n}]$ and 
$M(y^n,\chi_i)\simeq M(e,\chi_i)$ as $\DD_m$-modules via $\mm{u_i,n}\mapsto\mm{u_i}$.
 
 \smallskip
 
 \item[$\triangleright$] $M_\ell:=M(y^n,\rho_\ell)=\ku\{\mm{+,n,\ell},\mm{-,n,\ell}\}$ with 
 $M(y^n,\rho_\ell) = M(y^n,\rho_\ell)[y^{n}]$,
 \begin{align*}
 x\cdot\mm{\pm,n,\ell}=\mm{\mp,n,\ell}\quad\mbox{and}\quad y\cdot\mm{\pm,n,\ell}=\omega^{\pm\ell}\mm{\pm,n,\ell}.
 \end{align*}
\end{itemize}
Again, we write $\mm{\pm}=\mm{\pm,n,\ell}$ when both the parameters $\ell$ and $n$ are clear from the context.
The latter are the modules $M_\ell$ of \cite[\S2A1]{FG}.

\subsubsection{The modules $M(y^{i},\varrho)$}\label{subsec: the mikes} 
Let $1\leq i\leq n-1$. The conjugacy class of $y^i$ is $\{y^{i},y^{-i}\}$ and 
its centralizer $\cC_{y^{i}}$ is the subgroup $\langle y\rangle\simeq\Z_m$ 
whose simple representations are given by the characters
$\chi\_{k}(y)=\omega^k$ for $0\leq k\leq m-1$. So, the simple $\D\DD_m$-modules associated with $y^i$ are
\smallskip
\begin{itemize}
 \item[$\triangleright$] $M_{i,k}:=M(y^i,\chi\_{k})=\ku\{\mm{+,i,k},\mm{-,i,k}\}$ with $\mm{\pm, i,k}\in M_{i,k}[y^{\pm i}]$,
 \begin{align*}
 x\cdot\mm{\pm i,k}=\mm{\mp, i,k}\quad\mbox{and}\quad y\cdot\mm{\pm, i,k}=\omega^{\pm k}\mm{\pm, i,k}. 
 \end{align*} 
 \end{itemize}
Note that here the simple module is not concentrated in a single degree, in fact $\dim M_{i,k} = 2$ and 
$M_{i,k}= M_{i,k}[y^{-i}] \oplus M_{i,k}[y^{+i}]$. These are the modules $M_{i,k}$ of \cite[\S2A2]{FG}.
As before, we simply write $\mm{\pm}=\mm{\pm, i,k}$ when the context allows us to simplify notation.
Note that the simple modules $M_{\ell}$ can be describe as $M_{n,\ell}$, where the elements
are concentrated in degree $y^n=y^{-n}$.

\begin{notation}\label{notation: Mik Mell}
Given $1\leq i\leq n$ and $0\leq k\leq m-1$, we set $M_{i,k}=M(y^i,\chi\_{k})$ for $i\neq n$, and  $M_{n,k}=M(y^n,\rho_k)$.
\end{notation}

\subsubsection{The modules $M(x,\varrho)$}\label{subsec:simples x} 
The conjugacy class of $x$ is $\{xy^{2j}\mid j\in\Z_n\}$ and its centralizer is given 
by the subgroup $\langle x\rangle\oplus\langle y^n\rangle\simeq\Z_2\oplus\Z_2$. The  
irreducible representations are given by the characters 
$\sgn^s\ot\sgn^t$, $s,t\in\Z_2$, where $\sgn(x)=\sgn(y^n)=-1$ are the corresponding $\sgn$ representation
of the $\Z_{2}$ summand. 
Hence, the simple $\D\DD_m$-modules are
\smallskip
\begin{itemize}
 \item[$\triangleright$] $M_{0,s,t}:=M(x,\sgn^s\ot\sgn^t)=\ku\{\mm{j,0,s,t}\mid j\in\Z_n\}$ with 
 \begin{align*}
 &x\cdot\mm{0,0,s,t}=\sgn^s(x) \mm{0,0,s,t},&
 \\
 &x\cdot\mm{j,0,s,t}=\sgn^s(x)\sgn^t(y^n) \mm{n-j,0,s,t},\quad \text{ for all }j\neq0,&
 \\
 &y\cdot\mm{0,0,s,t}=\sgn^t(y^n) \mm{n-1,0,s,t},&
 \\
 &y\cdot\mm{j,0,s,t}=\mm{j-1,0,s,t},\quad \text{ for all }j\neq0,&\\
 &\text{and }\mm{j,0,s,t}\in M(x,\sgn^s\ot\sgn^t)[xy^{2j}],\quad \text{ for all }j\in\Z_n.&
 \end{align*}
\end{itemize}
In particular, $\dim M_{0,s,t}=n$ and $M_{0,s,t} = \bigoplus_{j\in\Z_n}M_{0,s,t}[xy^{2j}] $ 
as $\DD_{m}$-graded module.
We write $\mm{j}=\mm{j,0,s,t}$ when the notation is clear from the context.

\begin{figure}[!h]
\begin{tikzpicture}[scale=.9,every node/.style={scale=0.9}]

\node (M0st) at (0,0) {$M_{0,s,t}$};
\node (D) at (0,-4) {$\DD_m$-degree};

\draw[->] (M0st) to [out=135,in=180]
 (0,1) node [above,sloped] {$x\cdot$} to [out=00,in=45] (M0st);

\draw[->] (M0st) to [out=-135,in=180]
 (0,-1) node [below,sloped] {$y\cdot$} to [out=00,in=-45] (M0st);

\node (0) at (2,0) {$\mm{0}$};
\node (x) at (2,-4) {$x$};
\node (1) at (3.5,0) {$\mm{1}$};
\node (xy2) at (3.5,-4) {$xy^2$};
\node (2) at (5,0) {$\mm{2}$};
\node (xy4) at (5,-4) {$xy^4$};
\node     at (7,0) {$\cdots$};
\node (n2) at (8,0) {$\mm{\frac{n}{2}}$};
\node (xyn) at (8,-4) {$xy^n$};
\node     at (9,0) {$\cdots$};
\node (n-2) at (11,0) {$\mm{n-2}$};
\node (xyn-2) at (11,-4) {$xy^{m-4}$};
\node (n-1) at (12.5,0) {$\mm{n-1}$};
\node (xyn-1) at (12.5,-4) {$xy^{m-2}$};

\draw[->] (0) to [out=135,in=180]
 (2,1) node [above,sloped] {$\sgn(x)^s$} to [out=00,in=45] (0);

\draw[<->] (1) to [out=45,in=135]  node [above,sloped] {$\sgn(x)^s\sgn(y^n)^t$} (n-1);
\draw[<->] (2) to [out=45,in=135] node [above,sloped] {\tiny$\sgn(x)^s\sgn(y^n)^t$}(n-2) ;
\draw[->] (n2) to [out=135,in=180]
 (8,1) node [above,sloped,scale=.75] {\tiny$\sgn(x)^s\sgn(y^n)^t$} to [out=0,in=45] (n2);

\draw[->] (2,-.3) to [out=-45,in=-135]   node [below,sloped] {$\sgn(y^n)^t$} (12.5,-.3);

\draw[->] (n-1) to [out=225,in=-45] (n-2); 

\draw[->,dashed] (n-2) to [out=225,in=0] (10,-.5); 

\draw[->,dashed] (6,-.5) to [out=180,in=-45] (2);

\draw[->] (2) to [out=225,in=-45] (1); 
\draw[->] (1) to [out=-135,in=-45] (0);

%

\draw[dotted] (M0st) -- (D);
\draw[dotted] (0) -- (x);
\draw[dotted] (1) -- (xy2);
\draw[dotted] (2) -- (xy4);
\draw[dotted] (n2) -- (xyn);
\draw[dotted] (n-2) -- (xyn-2);
\draw[dotted] (n-1) -- (xyn-1);

\end{tikzpicture}
\caption{The simple module $M_{0,s,t}$ associated with the conjugacy class of $x$.}
\label{fig:M0st}
\end{figure}
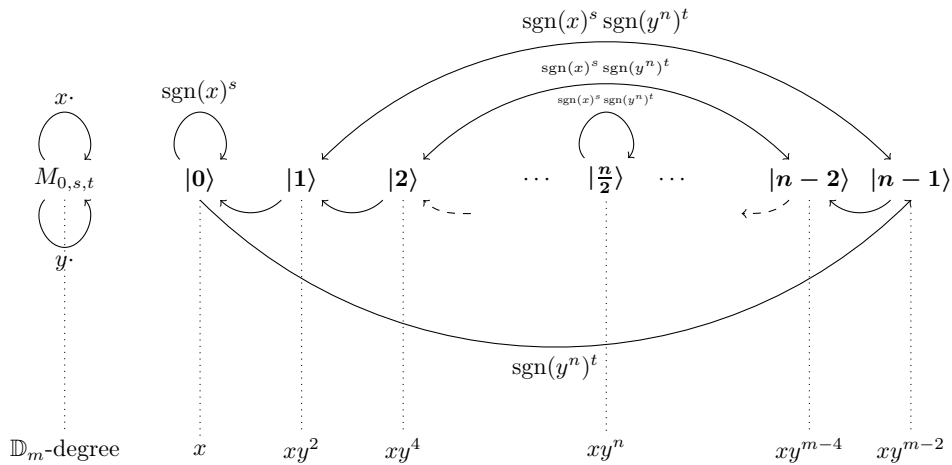

\subsubsection{The modules $M(xy,\varrho)$}\label{subsec:simples xy} 
The conjugacy class of $xy$ is $\{xy^{2j+1}\mid j\in\Z_n\}$. 
Its centralizer is the subgroup $\langle xy\rangle\oplus\langle y^n\rangle\simeq\Z_2\oplus\Z_2$ whose simple representations are 
given by the characters $\sgn^s\ot\sgn^t$, $s,t\in\Z_2$, where $\sgn(xy)=\sgn(y^n)=-1$ 
are the corresponding $\sgn$ representation
of the $\Z_{2}$ summand. 
Thus, the simple $\D\DD_m$-modules are
\smallskip
\begin{itemize}
\item[$\triangleright$] $M_{1,s,t}:=M(xy,\sgn^s\ot\sgn^t)=\ku\{\mm{j,1,s,t}\mid j\in\Z_n\}$ with
 \begin{align*}
 &y\cdot\mm{0,1,s,t}=\sgn^t(y^n) \mm{n-1,1,s,t},&
 \\
 &y\cdot\mm{j,1,s,t}=\mm{j-1,1,s,t},&
 \\
 &x\cdot\mm{j,1,s,t}=\sgn^s(xy)\sgn^t(y^n) \mm{n-j-1,1,s,t},&\quad& \text{for all } j\in\Z_n,&\\
 &\text{and } \mm{j,1,s,t}\in M(xy,\sgn^s\ot\sgn^t)[xy^{2j+1}], &\quad& \text{for all } j\in\Z_n.&
 \end{align*}
 \end{itemize}
 Here, $\dim M_{1,s,t}=n$ and $M_{1,s,t} = \bigoplus_{j\in\Z_n}M_{1,s,t}[xy^{2j+1}] $ 
as $\DD_{m}$-graded module.
Eventually, we might simply write $\mm{j}=\mm{j,1,s,t}$.

\begin{figure}[!h]
\begin{tikzpicture}[scale=.9,every node/.style={scale=0.9}]

\node (M1st) at (0,0) {$M_{1,s,t}$};
\node (D) at (0,-4) {$\DD_m$-degree};

\draw[->] (M1st) to [out=135,in=180]
 (0,1) node [above,sloped] {$x\cdot$} to [out=00,in=45] (M1st);

\draw[->] (M1st) to [out=-135,in=180]
 (0,-1) node [below,sloped] {$y\cdot$} to [out=00,in=-45] (M1st);

\node (0) at (2,0) {$\mm{0}$};
\node (x) at (2,-4) {$xy$};
\node (1) at (3.5,0) {$\mm{1}$};
\node (xy2) at (3.5,-4) {$xy^3$};
\node (2) at (5,0) {$\mm{2}$};
\node (xy4) at (5,-4) {$xy^5$};

\node     at (7.35,0) {$\cdots$};

\node (n-3) at (9.5,0) {$\mm{n-3}$};
\node (xyn-3) at (9.5,-4) {$xy^{m-5}$};

\node (n-2) at (11,0) {$\mm{n-2}$};
\node (xyn-2) at (11,-4) {$xy^{m-3}$};
\node (n-1) at (12.5,0) {$\mm{n-1}$};
\node (xyn-1) at (12.5,-4) {$xy^{m-1}$};

\draw[<->] (0) to [out=45,in=135]  node [above,sloped] {$\sgn(x)^s\sgn(y^n)^t$} (n-1);

\draw[<->] (1) to [out=45,in=135]  node [above,sloped] {\tiny $\sgn(x)^s\sgn(y^n)^t$} (n-2);

\draw[<->] (2) to [out=45,in=135]  node [above,sloped,scale=.75] {\tiny $\sgn(x)^s\sgn(y^n)^t$} (n-3);
 
 \draw[->] (2,-.3) to [out=-45,in=-135]  node [below,sloped] {$\sgn(y^n)^t$} (12.5,-.3);
 
\draw[->] (n-1) to [out=225,in=-45] (n-2); 

\draw[->,dashed] (n-3) to [out=225,in=0] (8.5,-.5); 

\draw[->,dashed] (6,-.5) to [out=180,in=-45] (2); 

\draw[->] (1) to [out=225,in=-45] (0); 
\draw[->] (2) to [out=225,in=-45] (1); 
\draw[->] (n-2) to [out=225,in=-45] (n-3);

%

\draw[dotted] (M1st) -- (D);
\draw[dotted] (0) -- (x);
\draw[dotted] (1) -- (xy2);
\draw[dotted] (2) -- (xy4);
\draw[dotted] (n-3) -- (xyn-3);
\draw[dotted] (n-2) -- (xyn-2);
\draw[dotted] (n-1) -- (xyn-1);

\end{tikzpicture}
\caption{The simple module $M_{1,s,t}$ associated with the conjugacy class of $xy$.}
\label{fig:M1st}
\end{figure}
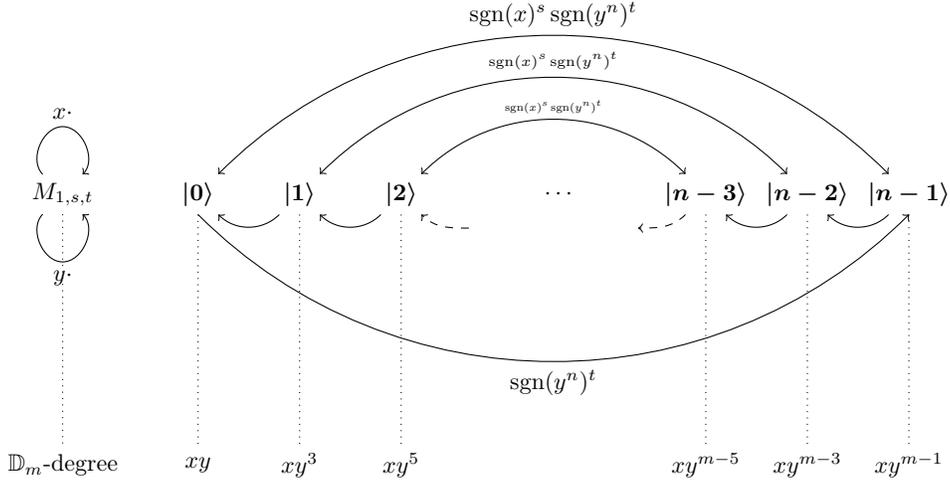

\subsection{Some tensor products of weights}\label{subsec:some tensor products} \

The category of $\D\DD_{m}$-modules is semi\-simple. As such, any tensor product of two weights
can be written as a direct sum of weights. In order to perform our study on simple modules
over doubles of bosonizations of Nichols algebras, which is carried out in 
\S\ref{sec:simple-MI-modules}, by
dealing with Verma modules, 
we need to know the direct summands of the following products of simple $\D\DD_m$-modules. 

In the following two lemmata, we decompose the tensor product of the 
simple modules $M_{r,s,t}$ with $r,s,t\in\Z_2$ as in \S \ref{subsec:simples x}, \S \ref{subsec:simples xy}
with some other families of simple modules.

\begin{lema}\label{le:mik ot mrst}
Let $M_{i,k}$ be a simple $\D\dm$-module  with 
 $1\leq i\leq n$, $0\leq k\leq m-1$ as in Notation \ref{notation: Mik Mell}.
 Then
 $$
 M_{i,k}\ot M_{r,s,t}\simeq M_{r+i,s+1+\delta_{i,n}t,t+k}\oplus M_{r+i,s+\delta_{i,n}t,t+k},
 $$
as $\D\DD_m$-modules; here we write $r+i$, $s+1+\delta_{i,n}t$ and $t+k$ for their classes in $\Z_{2}$. 
Moreover, the simple submodules inside the tensor product are given by  
$N_{\pm} = \Bbbk \{y^{a}\cdot \fn_{\pm} :\ 0\leq a\leq n-1\}$, where $\fn_{\pm}=\omega^{rk}\mm{-}\ot\mm{0}\pm\mm{+}\ot\mm{i}$, with $\mm{n}=\mm{0}$ if $i=n$, and  
\begin{align*}
N_+\simeq M_{r+i,s+1+\delta_{i,n}t,t+k}\,\mbox{ and }\, N_-\simeq M_{r+i,s+\delta_{i,n}t,t+k}.
\end{align*}
\end{lema}

\begin{proof}
We prove first the case $i<n$ and $r=0$, \textit{i.e.} $M_{0,s,t}$ is as in \S  \ref{subsec:simples x}.
Let us show that the subspaces $N_{\pm}$ are simple $\D\dm$-modules.
The elements $\fn_{\pm}$ are eigenvectors of $y^{n}$ with eingenvalue $\sgn^{t+k}(y^n)$ because
\begin{align*}
y^n\cdot(\mm{\pm}\ot\mm{0})&=(-1)^k\sgn^t(y^n)\,\mm{\pm}\ot\mm{0}
\end{align*}
and $\omega^{n}=\omega^{-n}=-1$. Also, it is straightforward to check that
\begin{align}\label{eq:x yi}
x\cdot \fn_{\pm}= \mp \sgn^{s}(x) y^{i}\cdot \fn_{\pm}. 
\end{align}
This implies that $N_{\pm} = \Bbbk \{y^{a}\cdot \fn_{\pm} :\ 0\leq a\leq n-1\}$ are $\DD_m$-submodules of $M_{i,k}\ot M_{r,s,t}$. 
Moreover, the elements $\fn_{\pm}$ are homogeneous of the same degree for
\begin{align*}
\deg(\mm{-}\ot\mm{0})&=\deg\mm{-}\deg\mm{0}=y^{-i}x=xy^{i}\mbox{ and}\\ 
\deg(\mm{+}\ot\mm{i})&=\deg\mm{+}\deg\mm{i}=y^{i}xy^{2i}=xy^{i}.
\end{align*}
Hence $\deg y^{a}\cdot \fn_{\pm} = y^{a}xy^{i}y^{-a} = xy^{i-2a}$ for 
all $0\leq a\leq n-1$. This implies that in fact  $N_{\pm}$ are $\D\dm$-modules with $\dim N_{\pm} = n$. 
Moreover, a direct check shows that they are isomorphic 
to the simple $\D\DD_m$-modules displayed in \S\ref{subsec:simples x} and \S\ref{subsec:simples xy}, 
depending on the parity of $i$. One way to distinguish these modules is by looking at the 
eigenvalues of the action of $y^n$ and $x$ 
on the homogeneous component of $\DD_m$-degree $x$ if $i=2z$ is even, or 
the action of $y^n$ and $xy$
on the homogeneous component of $\DD_m$-degree $xy$ if
$i=2z+1$ is odd. 
In the case of $N_{\pm}$, these homogeneous components are spanned by $y^z\cdot\fn_{\pm}$, respectively.

If $i=2z$ is even, then  
$x \cdot \big(y^{z} \cdot\fn_{\pm}\big) = 
\mp \sgn^{s}(x) \big(y^{z}\cdot \fn_{\pm}\big)$ by  \eqref{eq:x yi}. Hence $M_{0,s,t+k}\simeq N_-$ and $M_{0,s+1,t+k}\simeq N_+$. If $i=2z+1$ is odd, then 
$(xy) \cdot \big(y^{z} \cdot \fn_{\pm}\big) = 
\mp \sgn^{s}(x) \big(y^{z}\cdot \fn_{\pm}\big)$ by \eqref{eq:x yi}. Hence $M_{1,s,t+k}\simeq N_-$ and $M_{1,s+1,t+k}\simeq N_+$. 
In both cases the submodules are simple and non-isomorphic. Therefore $N_{+}\cap N_{-} = \{0\}$ and consequently, 
$M_{i,k}\ot M_{0,s,t}= N_{-}\oplus N_{+} \simeq M_{i,s,t+k}\oplus M_{i,s+1,t+k}$.

The strategy to prove the case $i<n$ and $r=1$ is similar. We still have that $y^n\cdot\fn_{\pm}= \sgn^{t+k}(y^n)\,\fn_{\pm}$ and, 
instead of \eqref{eq:x yi}, we have that $x\cdot\fn_{\pm} = \mp \sgn^{s}(xy) \big(y^{i+1}\cdot\fn_{\pm}\big)$. In this case, both $\fn_{\pm}$ are homogeneous of
degree $xy^{i+1}$. For $i$ even, we have $N_{-}\simeq  M_{1,s,t+k}$ and $N_{+}\simeq  M_{1,s+1,t+k}$, meanwhile for $i$ odd, we have
$N_{-}\simeq  M_{0,s,t+k}$ and $N_{+}\simeq  M_{0,s+1,t+k}$. We leave the details for the reader.

The proof for $i=n$ follows {\it mutatis mutandis} from the paragraphs above.
\end{proof}

We end this subsection with the following lemma.

\begin{lema}\label{le:chi2 ot x}
Let $M(e,\chi_2)$ be a simple $\D\dm$-modules as in 
 \S \ref{subsec:the mechis}. 
Then 
$$M(e,\chi_2)\ot M_{r,s,t}\simeq M_{r,s+1,t},$$
as $\D\DD_m$-modules, where we write $s+1$ for its class in $\Z_{2}$.
\end{lema}

\begin{proof} Straightforward. For instance, since $M(e,\chi_2) = M(e,\chi_2)[e]$ we have that  
$\bigg(M(e,\chi_2)\ot M_{r,s,t}\bigg)[xy^{2j+r}]= M(e,\chi_2)\ot \bigg(M_{r,s,t}[xy^{2j+r}]\bigg)$ for all $0\leq j\leq n-1$.
Also, as the action on $M(e,\chi_2)$ is given by $x\cdot\mm{u_2}=-\mm{u_2}$ and $y\cdot\mm{u_2}=\mm{u_2}$,
the lemma follows easily by the definition of the action on the tensor product.
\end{proof}

\subsection{Finite-dimensional Nichols algebras over Dihedral groups}\label{subsec:finite dim nichols}\ 

Here we recall the classification of finite-dimensional Nichols algebras in ${}_{\ku\DD_m}^{\ku\DD_m}\mathcal{YD}$, 
or equivalently in ${}_{\D\DD_m}\cM$. Roughly speaking,
they are all given by exterior algebras of direct sums of some families of simple $\D\dm$-modules $M_{i,k}$, recall Notation \ref{notation: Mik Mell}.

The classification in \cite[Theorem A]{FG} is given in terms of direct sums of three families of simple modules.
To shorten notation, we present them below in just one family by changing slightly the description.

\begin{notation}\label{notation:I}
Let $\cI$ be the family of all finite multisets $\{(i_1,k_1), ..., (i_r,k_r)\}$ of pairs such that 
$1\leq i_s\leq n$, $0\leq k_s\leq m-1$ and $\omega^{i_sk_t}=-1$ for all $1\leq s,t\leq r$. For $I\in\cI$, we define
$$
M_{I}= \bigoplus_{(i,k)\in I}M_{i,k}.
$$
\end{notation}

Observe that the families $\cI$, $\mathcal{L}$ and $\mathcal{K}$ defined in \cite{FG} fit in the description above. 
Indeed, if there is a pair $(n,\ell)$ in a sequence $I\in\cI$ and $(i,k)\in I$, then $\ell$ and $k$ must be odd 
because $\omega^{n\ell}=(-1)^{\ell}=-1=\omega^{nk}=(-1)^{k}$. 

\begin{theorem}\cite{FG}\label{thm:all-nichols-dm}
Let $\BV(M)$ be a finite-dimensional Nichols algebra
in ${}_{\D\DD_m}\mathcal{M}$. Then $M\simeq M_I$ for some $I\in\cI$ and $\BV(M)\simeq \bigwedge M$.\qed
\end{theorem}

\begin{obs}\label{obs:V=WU then W=Z}
We stress that if $V=M_I=W\oplus U$ is decomposable, then it satisfies 
Remarks \ref{obs:Z igual W} and \ref{obs:Z igual W -- homogeneous}. In particular,
$\BV(V)\#\ku\dm\simeq\BV(W)\#(\BV(U)\#\ku\dm)$ and $Z=W=\fL_H^U(W)$.
\end{obs}

\subsection{Drinfeld doubles of bosonizations of Nichols algebras over \texorpdfstring{$\dm$}{Dm}}
\label{subsec:drinfeld-double-dm}\

Here we present by generators and relations the Drinfeld double of the 
bosonization of a finite-dimensional Nichols algebra over $\dm$ by specifying the recipe given in 
\S\ref{subsub:generalized qg}. For that purpose, we need to set up some notation.

Let $V$ be a $\D\DD_m$-module with
$\dim \BV(V)<\infty$. We write
\begin{align*}
\D(V)=\D(\BV(V)\#\ku\DD_m) 
\end{align*}
and $\oV$ for the dual object of $V$ as in \S \ref{subsec:Drinfeld double}. 
By Theorem \ref{thm:all-nichols-dm}, we can fix a decomposition $V=\bigoplus_{(i,k)\in I}M_{i,k}$ and the orthogonal decomposition 
$\oV=\bigoplus_{(i,k)\in I}\overline{M_{i,k}}$. 
Given a (two-dimensional) 
direct summand $M_{i,k}$, we write $v_{+}$ and $v_{-}$ the elements of the basis 
$\{\mm{+},\mm{-}\}$ given in \S\ref{subsec: the mels} or \S\ref{subsec: the mikes}, as appropriate. 
So, we have
\begin{align*}
x\cdot v_{\pm}&=v_{\mp},\quad 
y\cdot v_{\pm}=\omega^{\pm k} v_{\pm}
\quad\mbox{for}\quad v_{\pm}\in M_{i,k}[y^{\pm i}].
\end{align*}
Also, we denote by $\alpha_+$ and $\alpha_-$ the elements in 
$\overline{M_{i,k}}$ satisfying
$\langle\alpha_{\bullet},v_{\circ}\rangle=\delta_{\bullet,\circ}$ where $\bullet,\circ\in\{+,-\}$. 
That is, $\{\alpha_\pm\}$ and $\{v_\pm\}$ are dual bases. 
Then the action of $\D\DD_m$ on these elements is determined by
\begin{align*}
x\cdot\alpha_{\pm}&=\alpha_{\mp},\quad 
y\cdot\alpha_{\pm}=\omega^{\mp k}\alpha_{\pm}
\quad\mbox{ for }\quad 
\alpha_{\pm}\in \overline{M_{i,k}}[y^{\mp i}].
\end{align*}
Thus, $\overline{M_{i,k}}\simeq M_{i,k}$ as $\D\DD_m$-modules 
via the assignment $\alpha_{\pm}\mapsto v_{\mp}$. 
The $\D\DD_m$-coactions defined by the functors $F_R$ and $F_{R^{-1}}$ on $M_{i,k}$ 
and $\overline{M_{i,k}}$, recall \eqref{eq:coactionsD(H)}, are
\begin{align}\label{eq:coactions}
(v_{\pm})\_{-1}\ot (v_{\pm})\_{0}&=y^{\pm i}\ot v_{\pm}\quad\mbox{and}\\
\notag
(\alpha_{\pm})\_{-1}\ot (\alpha_{\pm})\_{0}&=
\sum_{s=0}^{m-1}\omega^{\mp sk}\delta_{y^{-s}}
\,\ot\,\alpha_{\pm}+\omega^{\mp sk}\delta_{y^{-s}x}
\,\ot\,\alpha_{\mp}.
\end{align}

\begin{prop}
As an 
algebra, $\D(V)$ is generated by the elements of $V$, $\oV$, $\DD_m$ and $\ku^{\DD_m}$ subject to the relations \eqref{eq:relations of DDDm}--\eqref{eq:relations exterior} below.
\begin{itemize}
 \item For $s,t\in\dm$ and $z\in V\cup\oV$,
 \begin{align}
\label{eq:relations of DDDm}\delta_{tst^{-1}}\,t&=t\,\delta_{s}\\
\label{eq:relations bosonization}
  t\,z=(t\cdot z)\,t,
  &\quad\delta_t\,z=\sum_{s\in\DD_m}(\delta_s\cdot z)\,\delta_{s^{-1}t}.
\end{align}
\item For $v_{\pm}\in M_{i,k}$ and $\alpha_{\pm}\in\overline{M_{i,k}}$,
\begin{align}\label{eq:a+v+}
\alpha_+v_+&=-v_+\alpha_++\Phi_{++}\quad\mbox{with}\quad\Phi_{++}=1-\sum_{s=0}^{m-1}\omega^{sk}\delta_{y^{s}}y^{i},\\
\label{eq:a+v-}
\alpha_+v_-&=-v_-\alpha_++\Phi_{+-}
\quad\mbox{with}\quad\Phi_{+-}=
-\sum_{s=0}^{m-1}\omega^{-sk}\delta_{xy^{s}}y^{-i},
\\
\label{eq:a-v+}
\alpha_-v_+&=-v_+\alpha_-+\Phi_{-+}
\quad\mbox{with}\quad\Phi_{-+}
=-\sum_{s=0}^{m-1}\omega^{sk}\delta_{xy^{s}}y^{i},
\\
\label{eq:a-v-}
\alpha_-v_-&=
-v_-\alpha_-+\Phi_{--}
\quad\mbox{with}\quad\Phi_{--}
=1-\sum_{s=0}^{m-1}\omega^{-sk}\delta_{y^{s}}y^{-i}.
\end{align}
\item If $z,w\in V$, $z,w\in\oV$ or $w\in V$ and $z\in \oV$ are in orthogonal direct summands,
\begin{align}\label{eq:relations exterior}
z\,w&=-w\,z.
\end{align}
\end{itemize}
\end{prop}

\begin{proof}
We briefly explain why these relations hold. 
Relation \eqref{eq:relations of DDDm} is the commutation rule in $\D\dm$. 
The commutation rules \eqref{eq:relations bosonization} are given by the bosonizations 
$\BV(V)\#\D\dm$ and $\BV(\oV)\#\D\dm$.  The relations \eqref{eq:a+v+}--\eqref{eq:a-v-} and \eqref{eq:relations exterior} 
for generators in orthogonal direct summands follow from \eqref{eq:coactionsD(H)} and \eqref{eq:commutation rules gral} 
by using \eqref{eq:coactions}. By \eqref{eq:nichols de oV}, $\BV(\oV)$ 
is isomorphic as an algebra to a finite-dimensional Nichols algebra over $\dm$. 
Then it as an exterior algebra like $\BV(V)$ and hence \eqref{eq:relations exterior} holds. 
\end{proof}

Later on, in the upcoming section, we describe the simple $\D(V)$-modules using the strategy 
developed in \S\ref{subsec:simples} and \S\ref{subsec:decomposable}. 
Among all the relations above, we use only those involving $v_{\pm}$ and $\alpha_{\pm}$. 
Besides, the following elements of $\D(V)$ are going to be useful:
For a fixed a summand $M_{i,k}$, we set $\vt=v_+v_-$, $\at=\alpha_+\alpha_-$ and define 
\begin{align}\label{eq:Phi}
\Theta=-\Phi_{++}\Phi_{--}+\Phi_{+-}\Phi_{-+}\in\D\DD_m. 
\end{align}
Using \eqref{eq:a+v+}-\eqref{eq:a-v-}, a straightforward computation shows
that these elements satisfy \eqref{eq:discriminante} or its recursive version, as appropriate. Explicitly,
\begin{align}\label{eq:discriminante mik}
\at\vt-\Theta\in\oplus_{n>0}\BV^n(W)\ot\D(U)\ot\BV^n\left(\overline{W}\right)
\end{align}
where $W=M_{i,k}$ and $V=W\oplus U$, and $\D(U)=\D\DD_m$ if $V=W$.

\subsubsection{Spherical}

We finish this section by characterizing those Drinfeld doubles $\D(V)$ which are {\it spherical Hopf algebras}. 
This means by \cite[Definition 3.1]{BaW-adv} that $\D(V)$ has a group-like element $\varpi$ such that
\begin{align*}
\cS^2(h)=\varpi h\varpi^{-1}\quad\mbox{and}\quad\operatorname{tr}_{\fN}(\vartheta\varpi)=\operatorname{tr}_{\fN}(\vartheta\varpi^{-1})
\end{align*}
for all $h\in\D(V)$, $\fN\in{}_{\D(V}\cM$ and $\vartheta\in\operatorname{End}_{\D(V)}(\fN)$.
A group-like element satisfying the first condition is called {\it pivot} and, if it fulfills both conditions, it is called {\it spherical}. 
An involutive pivot, {\it i.e.} $\varpi^2=1$, is clearly an spherical element. The pivot is unique up to multiplication by a central group-like element.

\begin{theorem}\label{thm:DrinfeldDoubleSpherical}
The Drinfeld double $\D(V)$ is spherical if and only if $V$ does not contain a 
direct summand isomorphic to $M_{i,k}$ with both $i$ and $k$ even. 
In such a case, we may choose $\varpi=y^n\chi_3$ as the involutive spherical element.
\end{theorem}

\begin{proof}
By  \cite[Proposition 9]{Rad1}, we know that the group of group-like elements of $\D(V)$ equals
$\DD_m\times\{\chi_1,\chi_2,\chi_3,\chi_4\}$. 
Since $\cS^2$ is the identity on $\D\dm$, a pivot element has to belong to the subset $\{y^n\}\times\{\chi_1,\chi_2,\chi_3,\chi_4\}$, which 
consist only of
involutive elements. Then, in order to prove the statement, it is enough to analyse the existence of the pivot for $V$ simple.

Assume $V=M_{i,k}$ for some $1\leq i\leq n$ and $0\leq k< m$. 
Then
\begin{align*}
\chi_3\, v_{\pm}\, \chi_3 & =(-1)^{\pm i} v_{\pm},\quad &  \chi_3\, \alpha_{\pm}\, \chi_3 & =(-1)^{\mp i} \alpha_{\pm},\quad & \cS^2(v_\pm) & =-v_\pm,\\
y^n v_{\pm}y^n & =(-1)^{\pm k} v_{\pm},\quad & y^n\alpha_{\pm}y^n & =(-1)^{\mp k} \alpha_{\pm},\quad & \cS^2(\alpha_\pm) & =-\alpha_\pm,
\end{align*}
for the generators $v_{\pm}\in M_{i,k}$ and $\alpha_{\pm}\in\overline{M_{i,k}}$. 
Indeed, the formulas for the conjugation by $\chi_3$ and $y^n$ follow from \eqref{eq:relations bosonization}. 
The formulas for $\cS^2$ are deduced using \eqref{eq:coactions} and the definition of the coaction in a bosonization. 
Similarly, one can see that $\chi_{1}$ is central, $\chi_2$ commutes with $v_{\pm}$, and $\alpha_{\pm}$ and $\chi_3\chi_2=\chi_4$. 
We deduce then that $y^n\chi_3$ is a pivot if $i+k$ is odd and 
that there is no pivot when $i$ and $k$ are even. The case $i$ and $k$ both odd cannot occur because by assumption $\omega^{ik}=-1$.
\end{proof}

\begin{obs}
The quantum dimension of any simple module in ${}_{\D(V,\dm)}\cM$ is zero, except for those simple modules that are rigid.
\end{obs}



\section{Characters of simple \texorpdfstring{$\D(V)$}{D(V)}-modules}\label{sec:simple-MI-modules} 

In this section we follow the strategy summarized in \S\ref{subsec:simples} and \S\ref{subsec:decomposable}
to describe the simple modules over 
$\D(M_I)$ for $V=M_I=\oplus_{(i,k)\in I}M_{i,k}$ with
$I$ as in Notation \ref{notation:I}. Recall the set of weight $\Lambda$ in \S\ref{simple DDDm modules}. For $\lambda\in\Lambda$, we set
\begin{align*}
\fL^{I}(\lambda):=\fL^{M_{I}}_{\ku\dm}(\lambda),
\end{align*}
the simple highest-weight module over 
$\D(M_I)$ associated with $\lambda$. The appearance of $\fL^{I}(\lambda)$ depends 
on certain subsets of $\Lambda$ where the weight $\lambda$ belongs. 
We present first these subsets and then state the results. 

First, for $(i,k) \in I$, we fix the partition $\Lambda=\Lbr_{i,k}\cup\Lbp_{i,k}\cup\Lbo$ given in 
Table \ref{tab:Lbr Lbp Lbo Miks}. The subset $\Lbr_{i,k}$ corresponds to the rigid 
simple modules when $I=\{(i,k)\}$, that is, those weights
that satisfy $\fL^{(i,k)}(\lambda) = \lambda$ as $\D(M_{i,k})$-modules, see Lemma \ref{le:Lbr Lbp 1}.
The subset $\Lbp_{i,k}$ corresponds to the simple projective modules,
that is, those 
that satisfy $\fL^{(i,k)}(\lambda) = \fM^{(i,k)}(\lambda) := \fM^{M_{i,k}}_{\ku\dm}(\lambda)$ as 
$\D(M_{i,k})$-modules, see Lemma \ref{le:Lbr Lbp 1}.
 
\setlength{\extrarowheight}{5pt}
\begin{table}[h]
\begin{center}
\begin{tabular}{l|c|c|c}
& $\Lbr_{i,k}$ & $\Lbp_{i,k}$ & $\Lbo$ 
\\
\hline
$M(e,\chi_j)$ &  
\begin{tabular}{l}
$j=1,2$,\\
$j=3,4$ if $i$ even
\end{tabular}
&
\begin{tabular}{l}
\\
$j=3,4$ if $i$ odd
\end{tabular}
& --
\\
\hline
$M(e,\rho_\ell)$ & if $\omega^{i\ell}=1$ & if $\omega^{i\ell}\neq1$ & --
\\
\hline
$M(y^n,\chi_j)$ & 
\begin{tabular}{l}
$j=1,2$ if $k$ even,\\
$j=3,4$ if $i+k$ even
\end{tabular}
&
\begin{tabular}{l}
$j=1,2$ if $k$ odd,\\
$j=3,4$ if $i+k$ odd
\end{tabular}
& --
\\
\hline
$M_{p,q}$ & if $\omega^{iq+pk}=1$ & if $\omega^{iq+pk}\neq1$ & --
\\
\hline
$M_{r,s,t}$ & -- & -- &all
\\
\hline
\end{tabular}
\end{center}
\caption{Partition of the sets of weights with respect to $M_{i,k}$}
\label{tab:Lbr Lbp Lbo Miks}
\end{table}

The rigidity or projectivity of $\fL^{I}(\lambda)$ when $|I|>1$ is determined by the subsets defined below.
\begin{definition}\label{def:Ir and Ip}
For each $\lambda\in\Lambda$, we define 
\begin{align*}
I_\lambda^{\mathbf{r}}=\{(i,k)\in I\mid\lambda\in\Lbr_{i,k}\}\quad\mbox{and}\quad 
I_\lambda^{\mathbf{p}}=\{(i,k)\in I\mid\lambda\in\Lbp_{i,k}\}.
\end{align*} 
We also set $M_{I_\lambda^{\mathbf r}}=\bigoplus_{(i,k)\in I_\lambda^{\mathbf r}}M_{i,k}$ and 
$M_{I_\lambda^{\mathbf p}}=\bigoplus_{(i,k)\in I_\lambda^{\mathbf p}}M_{i,k}$. 
\end{definition}

For $\lambda \in \Lambda \setminus \Lbo$, we have that $M_I=M_{I_\lambda^{\mathbf p}}\oplus M_{I_\lambda^{\mathbf r}}$ as $\D\DD_m$-modules,
because $I = I_\lambda^{\mathbf p} \cup I_\lambda^{\mathbf r}$ and $\Lbr_{i,k}\cap\Lbp_{i,k}=\emptyset$ for all $(i,k)\in I$. In particular, 
we are under the hypothesis of \S\ref{subsec:decomposable}, with 
$U = M_{I_\lambda^{\mathbf r}}$ and $W=M_{I_\lambda^{\mathbf p}}$. 

\smallbreak

Here is our main result which, in particular, gives the characters of the simple $\D(M_I)$-modules. To
simplify the notation, we write the associated functors  
$\fL^{J}=\fL^{M_{J}}_{\ku\dm}$ and
$\fM^{K}_J=\fM^{M_{K}}_{\BV(M_{J})\#\Bbbk \dm}$ if $M_{I} = M_{J}\oplus M_{K}$.

\begin{theorem}\label{thm:simple-D(M_I)-modules}
Let $\lambda \in \Lambda$ and $M_I=\bigoplus_{(i,k)\in I}M_{i,k} \in {}_{\D\dm}\cM$ with $I\in\II$.
The simple highest-weight $\D(M_{I})$-modules are described as follows:
\begin{enumerate}
\item[$(a)$] If $\lambda\in\Lambda\setminus\Lbo$, then $\fL^{I_\lambda^{\mathbf r}}(\lambda) = \lambda$ is rigid as $\D(M_{I_\lambda^{\mathbf r}})$-modules and
\begin{align*}
\fL^{I}(\lambda)\simeq \fM^{I_\lambda^{\mathbf p}}_{I_\lambda^{\mathbf r}}(\fL^{I_\lambda^{\mathbf r}}(\lambda))
\end{align*}
as $\D(M_I)$-modules. In particular, $\Res\big(\fL^{I}(\lambda)\big) = \BV(M_{I_\lambda^{\mathbf p}})\ot \lambda$ and 
$\dim\fL^{I}(\lambda)=4^{|I_\lambda^{\mathbf p}|}\,\dim\lambda$. 

\smallbreak

\item[$(b)$] If $\lambda=M_{r,s,t}\in\Lbo$, then as graded $\D\DD_m$-modules,
$$
\Res\big(\fL^{I}(M_{r,s,t})\big)\simeq\bigoplus_{J\subseteq I}M_{r+i_J,s+\ell_J+\epsilon_J,t+k_J}[-|J|],
$$
where $i_J=\sum_{(i,k)\in J}i$, $k_J=\sum_{(i,k)\in J}k$, $\ell_J=\sum_{(n,\ell)\in J}\ell$ and $\epsilon_J=1$ if $\ell_J\neq0$ and zero otherwise; 
$i_{J}=k_{J}=\ell_J=0$ if $J=\emptyset$. In particular, $\dim\fL^{I}(M_{r,s,t})=2^{|I|}n$.
\end{enumerate}
\end{theorem}

\begin{proof}
With the aim of giving a clear exposition, we prove in detail the case in which $I$ does not contain pairs $(n,\ell)$. 
In particular, $\ell_{J}=0$ and $\epsilon_{J} = 0$.
The other case follows \textit{mutatis mutandis}.

We proceed by induction on the cardinal of $I$. 
The case $|I|=1$ is considered in \S\ref{subsec:MIsinglecase}, see Lemmata \ref{le:Lbr Lbp 1} and \ref{le:simples Lbo}. 
The inductive step is then proved in \S\ref{subsec:MIsemisimple}, see Lemma \ref{le:Lbr Lbp} for part $(a)$ 
and Lemma \ref{le:simple simple MIrlambda} for part $(b)$.
\end{proof}

As a direct consequence, one gets the description of the rigid and simple projective modules over $\D(M_{I})$. 

\begin{cor}
 Let $\lambda \in \Lambda$. Then
 \begin{enumerate}
  \item[$(a)$] $\fL^{I}(\lambda) = \lambda$ if and only if $I = I_\lambda^{\mathbf r}$. 
  \item[$(b)$] $\fL^{I}(\lambda) = \fM^{I}(\lambda)$ if and only if $I = I_\lambda^{\mathbf p}$.
 \end{enumerate}
\qed
\end{cor}

\subsection{The singleton case} \label{subsec:MIsinglecase} \

\smallbreak

We assume here that $I=\{(i,k)\}$ with $1\leq i\leq n-1$, $0\leq k\leq m-1$ and $\omega^{ik}=-1$.
We keep the notation of \S\ref{subsec:drinfeld-double-dm}. 
In particular, $V= M_{i,k}$, $v_{\pm}$ and $\alpha_{\pm}$ are the generators of $\D(M_{i,k})$ that belong to $M_{i,k}$ and its dual, respectively. 
The elements $\Phi_{\bullet,\circ}$ with $\bullet,\circ \in \{+,-\}$   
and $\Theta$ defined in \eqref{eq:a+v+}--\eqref{eq:a-v-} and \eqref{eq:Phi} are instrumental to determine which simple modules are rigid or projective.

\begin{lema}\label{le:Lbr Lbp 1} 
Let $\lambda \in \Lambda$ and $\fL^{I}(\lambda)$ be a simple $\D(M_{i,k})$-module. Then
\begin{enumerate}
\item[$(a)$] \label{le:Lbr i k}  
$\fL^{I}(\lambda)=\lambda$ if and only if $\lambda\in \Lbr_{i,k}$.
\item[$(b)$] \label{le:Lbp i k} 
$\fL^{I}(\lambda)=\fM^{I}(\lambda)$ if and only if $\lambda\in \Lbp_{i,k}$.
\end{enumerate}
\end{lema}

\begin{proof}
By looking at the $\DD_m$-degree, one easily sees that  
$\Phi_{\pm,\mp}\,\lambda=0$ for $\lambda\in\Lbr_{i,k}\cup\Lbp_{i,k}$, but it is non-zero for $\lambda\in\Lbo$. 
Besides, one may check that $\Phi_{\pm,\pm}\,\lambda=0$ for $\lambda\in\Lbr_{i,k}$, 
meanwhile it is non-zero for $\lambda\in\Lbp_{i,k}\cup\Lbo$. 
Hence $(a)$
follows from Remark \ref{obs:simple simple} for $V=M_{i,k}$. 

Analogously, through a sheer calculation 
one can show that $\Theta\lambda\neq0$ for $\lambda \in\Lbp_{i,k}$ and $\Theta\lambda=0$ for $\lambda\in\Lbo$. 
Thus, $(b)$
follows from Remark \ref{obs:verma simple} for $V=M_{i,k}$.
\end{proof}

For the remaining simple modules, we proceed as in \S\ref{subsub:application 2}.

\begin{lema}\label{le:simples Lbo}
Let $\lambda=M_{r,s,t}\in\Lbo$. Then, as graded $\D\DD_m$-modules,
\begin{align*}
\Res\big(\fL^{I}(M_{r,s,t}))&\simeq M_{r,s,t}[0]\oplus M_{r+i,s,t+k}[-1],
\\
\Res\big(\fS^{I}(M_{r,s,t}))&\simeq M_{r+i,s+1,t+k}[-1]\oplus M_{r,s+1,t}[-2].
\end{align*}
Moreover,
 $\fL^{I}(M_{r,s,t})\simeq\fM^{I}(M_{r,s,t})/\fS^{I}(M_{r,s,t})$.
\end{lema}

\begin{proof} 
Taking into account (see Figure \ref{fig:M lambda Mik simple}) 
that $\BV(M_{i,k}) = \bigwedge M_{i,k} = \Bbbk \oplus M_{i,k} \oplus \Bbbk v_{top}$ with 
$\BV^{0}(M_{i,k}) = \ku$, $\BV^{1}(M_{i,k}) = M_{i,k}$, $\BV^{2}(M_{i,k}) = \Bbbk v_{top}$
and $\Bbbk v_{top} \simeq M(e,\chi_{2})$ as $\D\dm$-modules, 
we have that 
$\Res\big(\fM^{I}(M_{r,s,t})\big)$ is the direct sum of four weights. Indeed,
$\Res\big(\fM^{I}(M_{r,s,t})_{0}\big)= \Bbbk\ot M_{r,s,t} =M_{r,s,t}$, 
$\Res\big(\fM^{I}(M_{r,s,t})_{-1}\big) = M_{i,k}\ot M_{r,s,t} \simeq M_{i+r,s+1,t+k}\oplus M_{i+r,s,t+k}$
 by Lemma \ref{le:mik ot mrst}, and $ \Res\big(\fM^{I}(M_{r,s,t}) _{-2}\big)=\Bbbk v_{top} \ot M_{r,s,t} \simeq 
M(e,\chi_{2}) \ot M_{r,s,t} \simeq M_{r,s+1,t}$ by Lemma \ref{le:chi2 ot x}. Note that all these weights are in $\Lbo$. 
Then $\fM^{I}(M_{r,s,t})$ is not simple and its composition factors are not concentrated in a single degree 
by Lemma \ref{le:Lbr Lbp 1}, as they are not rigid and consist of more than a weight.

\begin{figure}[!h]
\begin{tikzpicture}
\begin{scope}[on background layer]
\draw[dotted] (-4.5,1) -- (5.9,1) node at (6.15,1) {$0$};

\draw[dotted] (-4.5,0) --  (6,0) node at (6,0) {$-1$} ;

\draw[dotted] (-4.5,-1) -- (6,-1) node at (6,-1) {$-2$};

\draw[->] (-3,.75) to [out=18,in=162]
 node [midway,above,sloped] {$\ot\lambda$} (2,.75);
\end{scope}

\node at (-6,0) {};

\begin{scope}[xshift=2.5cm,xscale=.6]
\fill[fill=gray!50] 
(-.15,-1.15) to [out=315,in=225] 
(.15,-1.15) to
[out=45,in=-45] 
(.15,-.85) --
(-.85,.15) to
[out=135,in=45]
(-1.15,.15) to
[out=225,in=135]
(-1.15,-.15) --
(-.15,-1.15);

\node[circle,fill=black,inner sep=0pt, minimum width=5pt,label=above:{$M_{r,s,t}$}] at (0,1) {};

%

\node[circle,fill=black,inner sep=0pt, minimum width=5pt,label={[fill=white]right:$M_{r+i,s,t+k}$}] at (1,0) {};

\node[circle,fill=black,inner sep=0pt, minimum width=5pt,label={[fill=white,label distance=4pt]left:$M_{r+i,s+1,t+k}$}] at (-1,0) {};

\node[circle,fill=black,inner sep=0pt, minimum width=5pt,label=below:{$M_{r,s+1,t}$}] at (0,-1) {};
\end{scope}

\begin{scope}[xshift=-3.5cm]

\node[circle,fill=black,inner sep=0pt, minimum width=5pt,label=above:{$\ku$}] at (0,1) {};

\node[circle,fill=black,inner sep=0pt, minimum width=5pt,label={[fill=white]left:$M_{i,k}$}] at (0,0) {};

\node[circle,fill=black,inner sep=0pt, minimum width=5pt,label=below:{$M(e,\chi_2)$}] at (0,-1) {};
\end{scope}

\end{tikzpicture}
\caption{The big dots represent the weights of $\BV(M_{i,k})$ and $\fM^{(i,k)}(\lambda)$. 
Their degrees are indicated on the right. Those in the shadow region form the socle $\fS^{(i,k)}(\lambda)$;
the others the head $\fL^{(i,k)}(\lambda)$.}
\label{fig:M lambda Mik simple}
\end{figure}
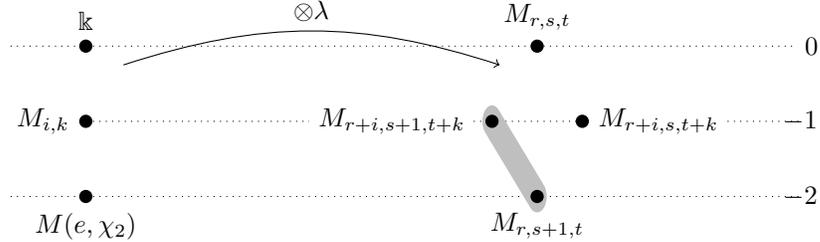

We deduce then that $\fM^{I}(M_{r,s,t})$ has exactly two composition factors, each of them has to be the direct 
sum of two weights. 
One must be the socle $\fS^{I}(M_{r,s,t})$ with $\Res\big(\fS^{I}(M_{r,s,t})\big)\simeq \mu\oplus(\ku\vt\ot M_{r,s,t})$ as $\D\DD_m$-modules, 
where $\mu$ is the unique highest-weight in degree $-1$ and $\fS^{I}(M_{r,s,t})\simeq\fL^{I}(\mu)[-1]$ by Corollary \ref{cor:highest weight of the socle}. 
The other composition factor is the head $\fL^{I}(M_{r,s,t})$ with $\fL^{I}(M_{r,s,t})\simeq\fM^{I}(M_{r,s,t})/\fS^{I}(M_{r,s,t})$ 
as $\D(M_{i,k})$-modules. 
Also,  
$\Res\big(\fL^{I}(M_{r,s,t})\big)\simeq M_{r,s,t}\oplus \overline{\lambda}$ as $\D\DD_m$-modules, where $\overline{\lambda}$
is the complement of $\mu$ in degree $-1$.

Hence, we should determine which weight in degree $-1$ is annihilated by $\oV =\overline{M_{ik}}$. By Lemma \ref{le:mik ot mrst},
we know these weights are generated by
$\fn_{\pm}=\omega^{rk}v_-\ot\mm{0}\pm v_+\ot\mm{i}$. Using \eqref{eq:a+v+}-\eqref{eq:a+v-}, we see that 
\begin{align}\label{eq: alpha+ on -1}
\alpha_+\fn_{\pm}=\alpha_+(\omega^{rk}v_-\ot\mm{0}\pm v_+\ot\mm{i})=-\mm{i}\pm\mm{i}. 
\end{align}
Then the action of $\overline{M_{ik}}$ on the weight generated by $\fn_-$ is non-trivial and
hence $\mu$ must be the weight generated by $\fn_+$; that is, $\mu =  M_{r+i,s+1,t+k}$. 
Therefore, we have that $\Res\big(\fS^{I}(\lambda)\big)\simeq M_{r+i,s+1,t+k}[-1]\oplus M_{r,s+1,t}[-2]$ and 
$\Res\big(\fL^{I}(\lambda)\big) \simeq M_{r,s,t}[0] \oplus M_{r+i,s,t+k}[-1]$.
\end{proof}

\begin{exa}
In Figure \ref{fig:L0st Mik par} below, we depict the simple module $\fL^{I}(M_{0,s,t})$ over $\D(M_{i,k})$ for $i$ even.
The nodes $\mm{j}$ denote the basis elements of the $\D\dm$-direct summands 
$M_{0,s,t}$ and $M_{0,s,t+k}$ of $\Res\big(\fL^{I}(M_{0,s,t})\big)$.
In each node, there should be two arrows going in and two arrows 
going out, but we only draw those corresponding to $\mm{\frac{i}{2}}$ in level $-1$ to make the diagram easy to read.
Keep the notation as in the proof of Lemma \ref{le:simples Lbo}, with 
$\fn_{\pm}=v_-\ot\mm{0}\pm v_+\ot\mm{i}$ and $\mm{0}, \mm{i} \in M_{0,s,t}$.
Set $\overline{\fn_{\pm}}=v_-\mm{0}\pm v_+\mm{i}$ for the images of these elements
in the quotient $\fM^{I}(M_{0,s,t})/\fS^{I}(M_{0,s,t})$.
Since $\overline{\fn_{+}}=0$, we have that $v_-\mm{0}=-v_+\mm{i}$ as elements of degree $-1$. 
Looking at the $\D\dm$-action,
one gets that both elements 
equal (a non-zero scalar multiple of) $\mm{\frac{i}{2}}$ in $M_{0,s,t+k}$ inside $\Res\big(\fL^{I}(M_{0,s,t})\big)$. 
This is depicted by the two arrows arriving at $\mm{\frac{i}{2}}$. Using \eqref{eq:a+v+} and \eqref{eq:a-v-}, one 
may see that $\alpha_+\mm{\frac{i}{2}}=-\mm{i}$ and $\alpha_-\mm{\frac{i}{2}}=\mm{0}$, respectively; these are the two arrows leaving $\mm{\frac{i}{2}}$.
\end{exa}

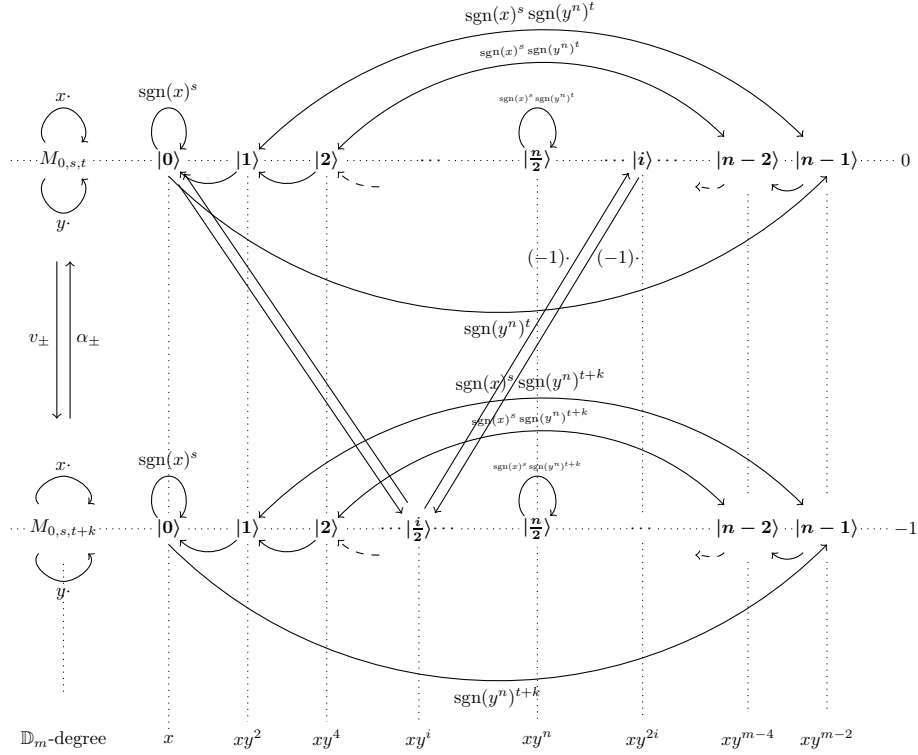
\begin{figure}[!h]
\begin{tikzpicture}[scale=.7,every node/.style={scale=0.7}]

\begin{scope}[yshift=7cm]
\draw[dotted] (-1,0) -- (16,0) node[circle,fill=white,inner sep=0pt, minimum width=5pt] {$0$};
\node[circle,fill=white,inner sep=0pt, minimum width=5pt] (M0st) at (0,0) {$M_{0,s,t}$};

\draw[->] (M0st) to [out=135,in=180]
 (0,1) node [above,sloped] {$x\cdot$} to [out=00,in=45] (M0st);

\draw[->] (M0st) to [out=-135,in=180]
 (0,-1) node [below,sloped] {$y\cdot$} to [out=00,in=-45] (M0st);

\node[circle,fill=white,inner sep=0pt, minimum width=5pt] (0) at (2,0) {$\mm{0}$};
\node[circle,fill=white,inner sep=0pt, minimum width=5pt] (1) at (3.5,0) {$\mm{1}$};
\node[circle,fill=white,inner sep=0pt, minimum width=5pt] (2) at (5,0) {$\mm{2}$};
\node[circle,fill=white,inner sep=0pt, minimum width=5pt]     at (7,0) {$\cdots$};
\node[circle,fill=white,inner sep=0pt, minimum width=5pt] (n2) at (9,0) {$\mm{\frac{n}{2}}$};
\node[circle,fill=white,inner sep=0pt, minimum width=5pt]     at (10.5,0) {$\cdots$};
\node[circle,fill=white,inner sep=0pt, minimum width=5pt]  (i)   at (11,0) {$\mm{i}$};
\node[circle,fill=white,inner sep=0pt, minimum width=5pt]     at (11.5,0) {$\cdots$};
\node[circle,fill=white,inner sep=0pt, minimum width=5pt] (n-2) at (13,0) {$\mm{n-2}$};
\node[circle,fill=white,inner sep=0pt, minimum width=5pt] (n-1) at (14.5,0) {$\mm{n-1}$};

\draw[->] (0) to [out=135,in=180]
 (2,1) node [above,sloped] {$\sgn(x)^s$} to [out=00,in=45] (0);

\draw[<->] (1) to [out=45,in=135]  node [above,sloped] {$\sgn(x)^s\sgn(y^n)^t$} (n-1);
\draw[<->] (2) to [out=45,in=135] node [above,sloped] {\tiny$\sgn(x)^s\sgn(y^n)^t$}(n-2) ;
\draw[->] (n2) to [out=135,in=180]
 (9,1) node [above,sloped,scale=.75] {\tiny$\sgn(x)^s\sgn(y^n)^t$} to [out=0,in=45] (n2);

\draw[->] (2,-.3) to [out=-45,in=-135]   node [below,sloped] {$\sgn(y^n)^t$} (14.5,-.3);

\draw[->] (n-1) to [out=225,in=-45] (n-2); 

\draw[->,dashed] (n-2) to [out=225,in=0] (12,-.5); 

\draw[->,dashed] (6,-.5) to [out=180,in=-45] (2);

\draw[->] (2) to [out=225,in=-45] (1); 
\draw[->] (1) to [out=-135,in=-45] (0);

\end{scope}

\draw[dotted] (-1,0) -- (16,0) node[circle,fill=white,inner sep=0pt, minimum width=5pt] {$-1$};

\node[circle,fill=white,inner sep=0pt, minimum width=5pt] (M0stk) at (0,0) {$M_{0,s,t+k}$};

\node[circle,fill=white,inner sep=0pt, minimum width=5pt] (D) at (0,-4) {$\DD_m$-degree};

\draw[->] (M0stk) to [out=135,in=180]
 (0,1) node [above,sloped] {$x\cdot$} to [out=00,in=45] (M0stk);

\draw[->] (M0stk) to [out=-135,in=180]
 (0,-1) node [below,sloped] {$y\cdot$} to [out=00,in=-45] (M0stk);

\node[circle,fill=white,inner sep=0pt, minimum width=5pt] (0k) at (2,0) {$\mm{0}$};
\node[circle,fill=white,inner sep=0pt, minimum width=5pt] (xk) at (2,-4) {$x$};
\node[circle,fill=white,inner sep=0pt, minimum width=5pt] (1k) at (3.5,0) {$\mm{1}$};
\node[circle,fill=white,inner sep=0pt, minimum width=5pt] (xy2k) at (3.5,-4) {$xy^2$};
\node[circle,fill=white,inner sep=0pt, minimum width=5pt] (2k) at (5,0) {$\mm{2}$};
\node[circle,fill=white,inner sep=0pt, minimum width=5pt]     at (6.25,0) {$\cdots$};
\node[circle,fill=white,inner sep=0pt, minimum width=5pt]  (i2k)   at (6.75,0) {$\mm{\frac{i}{2}}$};
\node[circle,fill=white,inner sep=0pt, minimum width=5pt]  (xyi2k)   at (6.75,-4) {$xy^i$};
\node[circle,fill=white,inner sep=0pt, minimum width=5pt] (xy4k) at (5,-4) {$xy^4$};
\node[circle,fill=white,inner sep=0pt, minimum width=5pt]     at (7.25,0) {$\cdots$};
\node[circle,fill=white,inner sep=0pt, minimum width=5pt] (n2k) at (9,0) {$\mm{\frac{n}{2}}$};
\node[circle,fill=white,inner sep=0pt, minimum width=5pt] (xynk) at (9,-4) {$xy^n$};
\node[circle,fill=white,inner sep=0pt, minimum width=5pt]     at (11,0) {$\cdots$};
\node[circle,fill=white,inner sep=0pt, minimum width=5pt]  (xyik)   at (11,-4) {$xy^{2i}$};
\node[circle,fill=white,inner sep=0pt, minimum width=5pt] (n-2k) at (13,0) {$\mm{n-2}$};
\node[circle,fill=white,inner sep=0pt, minimum width=5pt] (xyn-2k) at (13,-4) {$xy^{m-4}$};
\node[circle,fill=white,inner sep=0pt, minimum width=5pt] (n-1k) at (14.5,0) {$\mm{n-1}$};
\node[circle,fill=white,inner sep=0pt, minimum width=5pt] (xyn-1k) at (14.5,-4) {$xy^{m-2}$};

\begin{scope}[on background layer]
\draw[dotted] (M0stk) -- (D);
\draw[dotted] (0) -- (xk);
\draw[dotted] (1) -- (xy2k);
\draw[dotted] (2) -- (xy4k);
\draw[dotted] (i2k) -- (xyi2k);
\draw[dotted] (i) -- (xyik);
\draw[dotted] (n2) -- (xynk);
\draw[dotted] (n-2) -- (xyn-2k);
\draw[dotted] (n-1) -- (xyn-1k);
\end{scope}

\draw[->] (0k) to [out=135,in=180]
 (2,1) node [above,sloped] {$\sgn(x)^s$} to [out=00,in=45] (0k);

\draw[<->] (1k) to [out=45,in=135]  node [above,sloped] {$\sgn(x)^s\sgn(y^n)^{t+k}$} (n-1k);

\draw[<->] (2k) to [out=45,in=135] node [above,sloped] {\tiny$\sgn(x)^s\sgn(y^n)^{t+k}$}(n-2k) ;

\draw[->] (n2k) to [out=135,in=180]
 (9,1) node [above,sloped,scale=.75] {\tiny$\sgn(x)^s\sgn(y^n)^{t+k}$} to [out=0,in=45] (n2k);

\draw[->] (2,-.3) to [out=-45,in=-135]   node [below,sloped] {$\sgn(y^n)^{t+k}$} (14.5,-.3);

\draw[->] (n-1k) to [out=225,in=-45] (n-2k); 

\draw[->,dashed] (n-2k) to [out=225,in=0] (12,-.5); 

\draw[->,dashed] (6,-.5) to [out=180,in=-45] (2k);

\draw[->] (2k) to [out=225,in=-45] (1k); 
\draw[->] (1k) to [out=-135,in=-45] (0k);

%

\draw[transform canvas={xshift=-2.5pt},shorten <=1cm,shorten >=1cm,->] (M0st) -- node [left] {$v_{\pm}$} (M0stk);
\draw[transform canvas={xshift=2.5pt},shorten <=1cm,shorten >=1cm,<-] (M0st) -- node [right] {$\alpha_{\pm}$} (M0stk);

\draw[transform canvas={xshift=-2.5pt},shorten <=5pt,->] (0) -- (i2k);
\draw[transform canvas={xshift=2.5pt},shorten <=2pt,->] (i) -- node [pos=.25,right] {$(-1)\cdot$} (i2k);

\draw[transform canvas={xshift=2.5pt},shorten <=5pt,->] (i2k) -- (0);
\draw[transform canvas={xshift=-2.5pt},shorten <=2pt,->] (i2k) -- node [pos=.75,left] {$(-1)\cdot$} (i);

\end{tikzpicture}
\caption{$\fL^{(i,k)}(M_{0,s,t})$ over $\D(M_{i,k})$ for $i$ even.}
\label{fig:L0st Mik par}
\end{figure}

\subsection{The recursive step} \label{subsec:MIsemisimple}\

Fix a decomposition $M_I=\oplus_{(i,k)\in I}M_{i,k}$ with $|I|\geq1$ 
and keep the notation of \S\ref{subsec:drinfeld-double-dm}
As in the preceding subsection, we divide our analysis 
with respect to a partition on the set $\Lambda$. Here, we simply take $\Lambda=\Lbo \cup (\Lambda \setminus \Lbo)$.

The first lemma asserts that $\lambda\in\Lambda\setminus\Lbo $ is a rigid module 
over $\D(M_{I_\lambda^{\mathbf r}})$. 

\begin{lema}\label{le:simple simple MIrlambda}
Let $\lambda\in\Lambda\setminus\Lbo$. Then $\fL^{I_\lambda^{\mathbf r}}(\lambda)\simeq \lambda$, that is, $\lambda$
is a simple $\D(M_{I_\lambda^{\mathbf r}})$-module 
by letting $M_{I_\lambda^{\mathbf r}}$ and $\overline{M_{I_\lambda^{\mathbf r}}}$ act trivially.
\end{lema}

\begin{proof}
This follows by induction on the cardinal of $I_\lambda^{\mathbf r}$, using Lemma \ref{le:Lbr Lbp 1} $(a)$ 
and the recursive version of Remark \ref{obs:simple simple} with 
$M_{I_\lambda^{\mathbf r}}=M_{i,k}\oplus M_{I_\lambda^{\mathbf r}\setminus(i,k)}$. 
\end{proof}

Let $\lambda\in\Lambda\setminus\Lbo$. Using the decomposition 
$M_I=M_{I_\lambda^{\mathbf p}}\oplus M_{I_\lambda^{\mathbf r}}$ and the 
triangular decomposition \eqref{eq:triangular decomposition of DWV} associated with it, one may also 
consider the $\D(M_{I_\lambda^{\mathbf p}},M_{I_\lambda^{\mathbf r}},\dm)$-module
given by the description \eqref{eq:verma prima}, this is  
$$
\fM^{I_\lambda^{\mathbf p}}_{I_\lambda^{\mathbf r}}(\fL^{I_\lambda^{\mathbf r}}(\lambda))
= \fM^{I_\lambda^{\mathbf p}}_{I_\lambda^{\mathbf r}}(\lambda) =
\D(M_{I_\lambda^{\mathbf p}},M_{I_\lambda^{\mathbf r}},\dm)\ot_{\D(M_{I_\lambda^{\mathbf r}},\dm)\ot\BV(\overline{M_{I_\lambda^{\mathbf p}}})}\,\lambda.
$$

Then, $\fM^{I_\lambda^{\mathbf p}}_{I_\lambda^{\mathbf r}}(\lambda)$ admits a simple quotient $\fL^{I_\lambda^{\mathbf p}}_{I_\lambda^{\mathbf r}}(\lambda)$ 
which is isomorphic to $\fL^{I}(\lambda)$ by Theorem \ref{thm:simples W op V}.
In particular, there exists an
epimorphism of $\D(M_I)$-modules
\begin{align}\label{eq:epi Lbr Lbp}
\fM^{I_\lambda^{\mathbf p}}_{I_\lambda^{\mathbf r}}(\lambda)\twoheadrightarrow\fL^{I}(\lambda).
\end{align}

One could use the recursive version of Remark \ref{obs:verma simple} to prove that it is actually an isomorphism,
by showing that $\fM^{I_\lambda^{\mathbf p}}_{I_\lambda^{\mathbf r}}(\lambda) = \fL^{I_\lambda^{\mathbf p}}_{I_\lambda^{\mathbf r}}(\lambda)$ is simple. To do so, 
it would be enough to compute the element $\Theta$ for $M_{I_\lambda^{\mathbf p}}$. 
To avoid the long computation, we prove the 
simplicity of $\fM^{I_\lambda^{\mathbf p}}_{I_\lambda^{\mathbf r}}(\lambda)$
by induction on 
$|I_\lambda^{\mathbf p}|$, using the $\Theta$ already computed in the singleton case.

\begin{lema}\label{le:Lbr Lbp}
Let $\lambda\in\Lambda\setminus\Lbo$. Then $\fL^{I}(\lambda)\simeq \fM^{I_\lambda^{\mathbf p}}_{I_\lambda^{\mathbf r}}(\lambda)$ as $\D(M_I)$-modules.
In particular, 
$\fL^{I}(\lambda)\simeq\BV(M_{I_\lambda^{\mathbf p}})\ot_{\D\DD_m}\lambda$ as 
$\BV(M_{I_\lambda^{\mathbf p}})\#\D\DD_m$-modules.
\end{lema}

\begin{proof}
We proceed by induction on the cardinal $|I_\lambda^{\mathbf p}|$. 
If $I_\lambda^{\mathbf p}=\emptyset$, this is Lemma \ref{le:simple simple MIrlambda}.

Suppose $|I_\lambda^{\mathbf p}|\geq 1$. Let $(i,k)\in I_\lambda^{\mathbf p}$ and set $J=I\setminus\{(i,k)\}$. 
If $J=\emptyset $, then $I = \{(i,k)\} = I_\lambda^{\mathbf p}$, so
$\fM^{I_\lambda^{\mathbf p}}(\lambda) \simeq \fM^{I}(\lambda)$, and by Lemma \ref{le:Lbr Lbp 1} $(b)$ it follows that 
$\fM^{I_\lambda^{\mathbf p}}(\lambda) = \fL^{I}(\lambda)$.
Hence, we may assume that $J\neq \emptyset $. Since $|J_\lambda^{\mathbf p}| < |I_\lambda^{\mathbf p}|$, 
by the inductive hypothesis we have that
$$
\fL^J(\lambda)\simeq \fM^{J_\lambda^{\mathbf p}}_{J_\lambda^{\mathbf r}}(\lambda) = \fM^{J_\lambda^{\mathbf p}}_{J_\lambda^{\mathbf r}}(\fL^{J_\lambda^{\mathbf r}}(\lambda)) =
\D(M_J,\dm)\ot_{\D(M_{J_\lambda^{\mathbf r}},\dm)\ot\BV(\overline{M_{J_\lambda^{\mathbf p}}})}\,\lambda.
$$

Consider the generators $v_{\pm} \in M_{i,k}$ and $\alpha_{\pm} \in \overline{M_{i,k}}$ of $\D(M_I,\dm)$ 
defined as in \S\ref{subsec:drinfeld-double-dm}. 
Then, the commuting relations \eqref{eq:a+v+} -- \eqref{eq:a-v-} hold in $\D(M_I)$, 
and we may consider the corresponding
element $\Theta$ as in \eqref{eq:Phi} satisfying \eqref{eq:discriminante mik}. 

Furthermore, the action of $\Theta$ on $\fL^{J}(\lambda)$ 
is non-trivial because,
by the proof of Lemma \ref{le:Lbr Lbp 1} $(b)$, the action of $\Theta$ on 
$\lambda\subset\fL^J(\lambda)$ is not so. 
Then, by Theorem \ref{thm:simples W op V}  and the recursive version of Remark \ref{obs:verma simple} 
with $W=M_{i,k}$ and $U=M_J$, we have that
\begin{align*}
\fL^{I}(\lambda) \simeq \fL^{(i,k)}_J(\fL^{J}(\lambda)) \simeq \fM^{(i,k)}_J(\fL^{J}(\lambda)) &\simeq 
 \fM^{(i,k)}_{J_\lambda^{\mathbf p}}(\fM^{J_\lambda^{\mathbf p}}(\lambda))\simeq\fM^{I_\lambda^{\mathbf p}}(\lambda),
\end{align*}
where the last isomorphism follows by Lemma \ref{le:M-composicion}.
\end{proof}

The following lemma gives the description of $\fL^{I}(\lambda)$ for $\lambda\in \Lbo$. With it, we finish 
the proof of Theorem \ref{thm:simple-D(M_I)-modules}. Its
proof also relies on the recursive argument in \S\ref{subsec:decomposable}. 
More explicitly, this fits in the situation of Remark \ref{obs:Z igual W -- homogeneous} and \S\ref{subsub:application 2n}.

\begin{lema}\label{le:LMrst semisimple}
Let $\lambda=M_{r,s,t}\in\Lbo$. Then, as graded $\D\DD_m$-modules
$$
\Res\big(\fL^{I}(\lambda)\big)\simeq\bigoplus_{J\subseteq I}M_{r+i_J,s,t+k_J}[-|J|],
$$
where $i_J=\sum_{(i,k)\in J}i$ and $k_J=\sum_{(i,k)\in J}k$, with $i_{J}=0=k_{J}$ if $J=\emptyset$.
\end{lema}

\begin{proof}
We proceed by induction on $|I|$. If $|I|=1$, this is Lemma \ref{le:simples Lbo}.

Assume now $I=\{(i,k)\}\cup E$, then  $M_{I} = M_{i,k} \oplus M_{E}$. 
Following the strategy on \S\ref{subsub:application 2n}, taking $U=M_{E}$ and $W=M_{i.k}$, we 
describe $\fL^{I}(\lambda)$ as a quotient of the induced module
$$
\fM^{(i,k)}_E(\fL^{E}(\lambda)) = \fM^{(i,k)}_{\BV(M_{E})\#\ku\dm}(\fL^{E}(\lambda))= \D(M_{i,k}\oplus M_{E})\ot_{\D(M_{E})\ot \BV(\overline{M_{i,k}})}\fL^{E}(\lambda).
$$
Moreover, by the recursive version of \eqref{eq:verma iso}, we have
an isomorphism of $\D(M_{E})$-modules (here $\Res$ is the restriction functor to $\D(M_{E})$)
\begin{align}\label{eq:auxiliar 2 LMrst semisimple}
\Res\big(\fM^{(i,k)}_E(\fL^{E}(\lambda))\big) \simeq\bigoplus_{j=0}^2\BV^{j}(M_{i,k})\ot\fL^{E}(\lambda).
\end{align}

We first show that $\fL^{I}(\lambda)$ is not isomorphic to neither $\fL^{E}(\lambda)$ nor $\fM^{(i,k)}_E(\fL^{E}(\lambda))$. 
To do this, we consider the generators $v_{\pm}\in M_{i,k}$ and 
$\alpha_{\pm} \in \overline{M_{i,k}}$ of $\D(M_I)$ as in \S\ref{subsec:drinfeld-double-dm}. Thus, the commuting relations
\eqref{eq:a+v+} -- \eqref{eq:a-v-} hold and the corresponding element $\Theta$ as in \eqref{eq:Phi} satisfies \eqref{eq:discriminante mik}. 
Now, the action of some $\Phi_{\pm,\pm}$ is non-zero on $\lambda$ by the proof of Lemma \ref{le:Lbr Lbp 1} $(a)$ 
and hence $\fL^{I}(\lambda)\not\simeq\fL^{E}(\lambda)$ by the recursive version of Remark \ref{obs:simple simple}. 
Also, by the inductive hypothesis and Lemma \ref{le:mik ot mrst}, $\fL^{E}(\lambda)$ is the direct sum of weights in $\Lbo$. 
By Lemma \ref{le:Lbp i k}, the action of $\Theta$ on them is trivial and hence $\fL^{I}(\lambda)\not\simeq\fM^{(i,k)}_E(\fL^{E}(\lambda))$ 
by the recursive version of
Remark \ref{obs:verma simple}. 

We next analyse the structure of $\fM^{(i,k)}_E(\fL^{E}(\lambda))$ as $\D(M_{E})$-module, similarly as we did in Lemma \ref{le:simples Lbo}. See Figure \ref{fig:M lambda Mik decomposable}.

\begin{claim}
$\fM^{(i,k)}_E(\fL^{E}(\lambda))$ is semisimple as $\D(M_{E})$-module and  
\begin{align*}
\Res\big(\fM^{(i,k)}_E(\fL^{E}(\lambda))\big)\simeq&\ \fL^{E}(\lambda)\oplus
\fL^{E}(M_{r+i,s+1,t+k})\oplus\fL^{E}(M_{r+i,s,t+k})\oplus\fL^{E}(M_{r,s+1,t}),
\end{align*}
where the gradation induced by the new triangular decomposition is
\begin{align*}
\Res\big(\fM^{(i,k)}_E(\fL^{E}(\lambda))\big)_{0} &= \fL^{E}(\lambda)=  \fL^{E}(M_{r,s,t}),\\
\Res\big(\fM^{(i,k)}_E(\fL^{E}(\lambda))\big)_{-1} & = \fL^{E}(M_{r+i,s+1,t+k})\oplus\fL^{E}(M_{r+i,s,t+k})\\ 
\Res\big(\fM^{(i,k)}_E(\fL^{E}(\lambda))\big)_{-2} & = \fL^{E}(M_{r,s+1,t}). 
\end{align*}
\end{claim}

Indeed, $\BV^{0}(M_{i,k}) =\Bbbk \simeq M(e,\chi_1)$, $\BV^{1}(M_{i,k}) = M_{i,k}$
and $\BV^2(M_{i,k})\simeq M(e,\chi_2)$ are all rigid simple $\D(M_E)$-modules by the inductive hypotheses applied to $E$.
Then Theorem \ref{thm:rigid tensor} implies that
\begin{align*}
\BV^{0}(M_{i,k})\ot\fL^E(\lambda)&\simeq\fL^E(\lambda),\\
\BV^{1}(M_{i,k})\ot\fL^E(\lambda)&
\simeq\fL^E(M_{r+i,s+1,t+k})\oplus\fL^E(M_{r+i,s,t+k})\quad\mbox{(see Lemma \ref{le:mik ot mrst}),}\\
\BV^2(M_{i,k})\ot\fL^E(\lambda)&\simeq\fL^E( M_{r,s+1,t})\quad\mbox{(see Lemma \ref{le:chi2 ot x})}
\end{align*}
as $\D(M_E)$-modules. Therefore the claim follows from \eqref{eq:auxiliar 2 LMrst semisimple}.

\begin{figure}[!h]
\begin{tikzpicture}    

\draw[dotted] (-4.5,1) -- (5.9,1) node at (6.35,1) {$0$};

\draw[dotted] (-4.5,0) --  (6.25,0) node[fill=white] at (6.25,0) {$-1$} ;

\draw[dotted] (-4.5,-1) -- (6.25,-1) node[fill=white] at (6.25,-1) {$-2$};

\draw[->] (-3,.75) to [out=18,in=162]
 node [midway,above,sloped] {$\ot\fL^{E}(M_{r,s,t})$} (2,.75);

\node at (-6.35,0) {};

\begin{scope}[xshift=-3.5cm]

\node[circle,fill=black,inner sep=0pt, minimum width=5pt,label=above:{$\ku$}] at (0,1) {};

\node[circle,fill=black,inner sep=0pt, minimum width=5pt,label={[fill=white]left:$M_{i,k}$}] at (0,0) {};

\node[circle,fill=black,inner sep=0pt, minimum width=5pt,label=below:{$M(e,\chi_2)$}] at (0,-1) {};
\end{scope}

\begin{scope}[xshift=2.5cm,xscale=.6]

\fill[fill=gray!50] 
(-.15,-1.15) to [out=315,in=225] 
(.15,-1.15) to
[out=45,in=-45] 
(.15,-.85) --
(-.85,.15) to
[out=135,in=45]
(-1.15,.15) to
[out=225,in=135]
(-1.15,-.15) --
(-.15,-1.15);

\node[circle,fill=black,inner sep=0pt, minimum width=5pt,label=above:{$\fL^E(M_{r,s,t})$}] at (0,1) {};

\node[circle,fill=black,inner sep=0pt, minimum width=5pt,label={[fill=white]right:$\fL^E(M_{r+i,s,t+k})$}] at (1,0) {};

\node[circle,fill=black,inner sep=0pt, minimum width=5pt,label={[fill=white,,label distance=4pt]left:$\fL^E(M_{r+i,s+1,t+k})$}] at (-1,0) {};

\node[circle,fill=black,inner sep=0pt, minimum width=5pt,label=below:{$\fL^E(M_{r,s+1,t})$}] at (0,-1) {};
\end{scope}

\end{tikzpicture}
\caption{The dots represent the simple $\D(M_E)$-summands of $\BV(M_{i,k})$ and 
$\fM^{(i,k)}_E(\fL^{E}(M_{r,s,t}))$. 
Their degrees are indicated on the right. Those in the shadow region form its socle and the others its head which is isomorphic to $\fL^{(i,k)}(M_{r,s,t})$.}
\label{fig:M lambda Mik decomposable}
\end{figure}
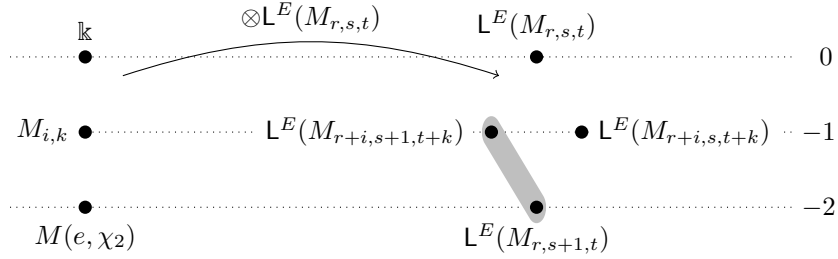

Being a quotient of $\fM^{(i,k)}_E(\fL^{E}(\lambda))$, $\fL^{I}(\lambda)$ is isomorphic to a direct sum of some of the simple 
$\D(M_{E})$-modules in the claim above. 

\begin{claim}\label{claim:verma LMrst semisimple}
As graded $\D(M_{E})$-modules, we have that 
\begin{align}\label{eq:final mik}
\Res\big(\fL^{I}(\lambda)\big)\simeq \fL^{E}(M_{r,s,t})\oplus
\fL^{E}(M_{r+i,s,t+k})[-1]
\end{align}
with the gradation induced by the new triangular decomposition.
\end{claim} 

Indeed, the claim follows by a counting argument similar to the one in the proof of Lemma \ref{le:simples Lbo}.
Using the decomposition \eqref{eq:auxiliar 2 LMrst semisimple}, one can
deduce that
$\Res\big(\fL^{I}(\lambda)\big)$ is isomorphic to $\fL^{E}(\lambda)\oplus
\fL^{E}(M_{r+i,\tilde s,t+k})[-1]$ as $\D(M_{E})$-module with either $\tilde{s}=s$ or $\tilde{s}=s+1$.
To determine which $\tilde{s}$ is, we can argue as in the proof of Lemma \ref{le:simples Lbo}, see \eqref{eq: alpha+ on -1}. 
Namely, it must be the simple $\D(M_{E})$-module over which the action of $\overline{M_{i,k}}$ is non-trivial. 
As in \eqref{eq: alpha+ on -1}, we see that $\overline{M_{i,k}}$ acts by zero on the weight $M_{r+i,s+1,t+k}$ 
and its action is non-zero on $M_{r+i,s,t+k}$. Then $\overline{M_{i,k}}$ must act trivially over the simple 
$\D(M_E)$-submodule generated by the former weight, 
that is $\fL^{E}(M_{r+i,s+1,t+k})$, and hence $\tilde{s}=s$.

Finally, the lemma follows by \eqref{eq:final mik} and the inductive hypothesis.
\end{proof}


\end{document}